\documentclass[11pt]{article}

\usepackage[english]{babel}
\usepackage[a4paper,top=3cm,bottom=3cm,left=2.5cm,right=2.5cm]{geometry}
\usepackage[utf8]{inputenc}
\usepackage[T1]{fontenc}
\usepackage[all,cmtip]{xy}
\usepackage{hyperref}

\usepackage{amsmath,amssymb,amsthm}	
\usepackage[mathscr]{eucal}			
\usepackage{amsmath,amssymb}
\usepackage{amsthm}
\usepackage{amsfonts}
\usepackage{enumerate}
\usepackage{mathabx}
\usepackage[all,cmtip]{xy}
\usepackage{amscd}

\usepackage[svgnames]{xcolor}
\usepackage{listings}

\usepackage{graphicx}
\usepackage{verbatim}

\usepackage{tikz}
\usepackage{tikz-cd}
\usetikzlibrary{arrows,calc,matrix,patterns,shadings,backgrounds,fit}

\usepackage{comment}
\usepackage{soul}

\theoremstyle{plain} 
\newtheorem{thm}{Theorem}[section]
\newtheorem{cor}[thm]{Corollary} 
\newtheorem{lemma}[thm]{Lemma} 
\newtheorem{prop}[thm]{Proposition}

\theoremstyle{definition} 
\newtheorem{dfn}[thm]{Definition}
\newtheorem{rmk}[thm]{Remark} 
\newtheorem{example}[thm]{Example} 
\newtheorem*{str}{Strategy of the proofs}
\newtheorem*{rw}{Related works and motivations}
\newtheorem*{plan}{Plan of the paper}
\newtheorem*{ack}{Acknowledgements}

\DeclareMathOperator{\C}{\mathbb{C}}
		
\DeclareMathOperator{\Z}{\mathbb{Z}}
\DeclareMathOperator{\R}{\mathbb{R}}

\DeclareMathOperator{\Hom}{\text{Hom}}

\DeclareMathOperator{\A}{\mathcal{A}}
\DeclareMathOperator{\T}{\mathcal{T}}
\DeclareMathOperator{\D}{\mathrm{D}^b}
\DeclareMathOperator{\G}{\text{Gr}}

\DeclareMathOperator{\ch}{\text{ch}}
\DeclareMathOperator{\td}{\text{td}}

\def\Ku{\mathop{\mathrm{Ku}}\nolimits}

\DeclareMathOperator{\Coh}{\text{Coh}}

\def\P{\ensuremath{\mathbb{P}}}
\def\L{\ensuremath{\mathbb L}}
\def\Hom{\mathop{\mathrm{Hom}}\nolimits}
\def\Forg{\mathop{\mathrm{Forg}}\nolimits}

\DeclareMathOperator{\CCoh}{\emph{Coh}}

\def\AA{\ensuremath{\mathcal A}}
\def\BB{\ensuremath{\mathcal B}}

\def\DD{\ensuremath{\mathcal D}}
\def\EE{\ensuremath{\mathcal E}}

\def\GG{\ensuremath{\mathcal G}}
\def\HH{\ensuremath{\mathcal H}}
\def\II{\ensuremath{\mathcal I}}

\def\MM{\ensuremath{\mathcal M}}
\def\NN{\ensuremath{\mathcal N}}
\def\OO{\ensuremath{\mathcal O}}
\def\PP{\ensuremath{\mathcal P}}

\def\TT{\ensuremath{\mathcal T}}
\def\UU{\ensuremath{\mathcal U}}

\def\ZZ{\ensuremath{\mathcal Z}}

\def\ch{\mathop{\mathrm{ch}}\nolimits}
\def\rk{\mathop{\mathrm{rk}}\nolimits}

\title{Stability conditions on Kuznetsov components of Gushel--Mukai threefolds and Serre functor}
\author{Laura Pertusi, Ethan Robinett}
\date{\today}

\begin{document}
\maketitle

\begin{abstract}
We show that the stability conditions on the Kuznetsov component of a Gushel--Mukai threefold, constructed by Bayer, Lahoz, Macrì and Stellari, are preserved by the Serre functor, up to the action of the universal cover of $\text{GL}^+_2(\R)$. As application, we construct stability conditions on the Kuznetsov component of special Gushel--Mukai fourfolds.      
\end{abstract} 

\section{Introduction} The idea of a stability condition on a triangulated category $\mathcal{T}$ was introduced by Bridgeland in \cite{Bridgeland07}, where he also proved that the space of all such stability conditions admits a natural topology making it a complex manifold, thus defining a novel invariant of triangulated categories. The study of the properties of this stability manifold is difficult, even if we consider only $\mathcal{T}= \D(X)$ for a smooth projective variety $X$ over $\mathbb{C}$, which was the setting that originally motivated Bridgeland. It is, for instance, unknown if this space is nonempty for all $X$ of dimension $\geq 3$, and a complete description of the stability manifold is known only when $X$ is a curve \cite{Bridgeland07, Macri, Okada}.  For more results in this direction, we refer the reader to \cite{Bridgeland, Arcara_Bertram} in the case of surfaces, and in higher dimension to \cite{BMT, Macri14, BMSZ, BMS, Li, Liu} and the references contained therein. 
\\
\indent A related line of inquiry is the study of the stability manifold of certain triangulated subcategories of $\D(X)$. The most well-studied situation of this kind is when $\D(X)$ admits an exceptional collection $E_0, \cdot \cdot \cdot, E_m$ and the subcategory $\mathcal{T}$ is $\langle E_0, \cdot \cdot \cdot, E_m \rangle^\perp$. When $X$ is a cubic fourfold, $\D(X)$ admits an exceptional collection given by $\mathcal{O}_X, \mathcal{O}_X(1), \mathcal{O}_X(2)$, and the right orthogonal is called the \textit{Kuznetsov component}, denoted by $\textrm{Ku}(X)$. Kuznetsov showed in \cite{Kuz_cubicfourfold} that for many cubic fourfolds (notably, in each case $X$ was rational), $\textrm{Ku}(X)$ is equivalent to the derived category of a K3 surface. Afterwards, it was shown in \cite{BLMS} that $\textrm{Ku}(X)$ admits stability conditions, and that the corresponding moduli spaces of semistable complexes are smooth, projective hyperkahler varieties \cite{BLMNPS}.
\\
\indent The richness of these results motivates the study of analogous situations. In particular, the derived categories of Gushel--Mukai (GM) varieties are known to contain exceptional collections \cite{KP}, the Kuznetsov components of which have been shown to admit stability conditions in \cite{BLMS, PPZ}. See Section \ref{sec_GM} for a summary. 

In this paper, we consider the case when $X$ is a GM threefold. In \cite[Theorem 6.9]{BLMS} Bayer, Lahoz, Macrì and Stellari constructed a family of stability conditions on $\textrm{Ku}(X)$, which we denote by $\sigma(\alpha,\beta)$, depending on two real numbers $\alpha$, $\beta$ satisfying certain conditions (see Theorem \ref{thm_BLMSresult} for a more precise formulation). Our main result concerns the action of the Serre functor $S_{\Ku(X)}$ of $\Ku(X)$ on the stability conditions $\sigma(\alpha,\beta)$. Recall that there is a right action of the universal cover $\widetilde{\textrm{GL}}_2^+(\mathbb{R})$ of the group $\textrm{GL}_2^+(\mathbb{R})$ of real $2 \times 2$ matrices with positive determinant on the stability manifold. We show that $S_{\textrm{Ku}(X)}[-2]$ acts as the identity of $\widetilde{\textrm{GL}}_2^+(\mathbb{R})$ on the stability conditions in the same orbit of the stability conditions $\sigma(\alpha, \beta)$ with respect to the $\widetilde{\textrm{GL}}_2^+(\mathbb{R})$-action. %We show that $S_{\textrm{Ku}(X)}$ preserves the orbit of the stability conditions $\sigma(\alpha, \beta)$ with respect to the $\widetilde{\textrm{GL}}_2^+(\mathbb{R})$-action. 

\begin{thm}[Theorem \ref{thm_mainsec3}, Corollary \ref{cor_finalresult}, Corollary \ref{cor_stabonfourfold}]
\label{thm_main}
Let $X$ be Gushel--Mukai threefold. Let $\sigma$ be a stability condition on the Kuznetsov component $\Ku(X)$ which is in the same orbit of $\sigma(\alpha, \beta)$ with respect to the action of $\widetilde{\mathrm{GL}}^+_2(\R)$. Then 
$$S_{\Ku(X)}[-2] \cdot \sigma = \sigma.$$
\end{thm}

\noindent Theorem \ref{thm_main} shows that the stability conditions constructed in \cite{BLMS} are Serre-invariant in the sense of Definition \ref{def_serreinvariant}. As pointed out in Theorem \ref{thm_fanoevengenus}, an analogous result holds more generally for certain Fano threefolds of Picard rank $1$, index $1$ and even genus. In this case, we show that the Serre functor of their Kuznetsov component preserves the orbit of the stability conditions constructed in \cite{BLMS} with respect to the $\widetilde{\textrm{GL}}_2^+(\mathbb{R})$-action. 
 
As an application, we construct stability conditions on the Kuznetsov component of a special GM fourfold. Recall that a special GM fourfold $X$ is a double cover of a linear section of the Grassmannian $\text{Gr}(2, 5)$ ramified over an ordinary GM threefold $Z$. By \cite[Corollary 1.3]{KP_cyclic} there is an exact equivalence 
$$\Ku(Z)^{\mathbb{Z}/ 2\mathbb{Z}} \simeq \Ku(X),$$
where $\Ku(Z)^{\mathbb{Z}/ 2\mathbb{Z}}$ denotes the category of $\mathbb{Z}/ 2\mathbb{Z}$-equivariant objects of $\Ku(Z)$ and the $\mathbb{Z}/ 2\mathbb{Z}$-action on $\Ku(Z)$ is given by $S_{\Ku(Z)}[-2]$.
\begin{thm}[Corollary \ref{cor_stabonfourfold}, Remark \ref{rmk_stabonfourfold}] \label{thm_application}
Let $X$ be a special GM fourfold and $Z$ be its associated ordinary GM threefold. Serre-invariant stability conditions on $\Ku(Z)$ induce stability conditions on the equivariant category $\Ku(Z)^{\mathbb{Z}/2\mathbb{Z}}$. In particular, they define stability conditions on $\Ku(X)$. 
\end{thm}

In Corollary \ref{cor_uniqueness} we show that there is a unique $\widetilde{\mathrm{GL}}^+_2(\R)$-orbit of Serre-invariant stability conditions on $\Ku(X)$ of a GM threefold $X$. 

\begin{rw}
In \cite[Proposition 5.7]{PY} it is shown that the stability conditions induced on the Kuznetsov component of a Fano threefold of Picard rank $1$ and index $2$ (e.g.\ a cubic threefold) with the method in \cite{BLMS} are Serre-invariant. Using this result, the authors further proved that non-empty moduli spaces of stable objects with respect to these stability conditions are smooth. They also gave another proof of the categorical Torelli Theorem for cubic threefolds in \cite[Theorem 5.17]{PY}, following the strategy in \cite[Theorem 1.1]{BMMS} where this result was proved for the first time (see also \cite{Soheyla} for a different approach). 

In fact, the property of Serre-invariance is very helpful in the study of the properties of moduli spaces and the stability of objects, see for instance \cite{JLLZ, LZ} for many recent applications. In \cite{FP} the notion of Serre-invariance is applied to show that the moduli space of stable Ulrich bundles of rank $d \geq 2$ on a cubic threefold is irreducible.

On the other hand, not all triangulated subcategories of the bounded derived category of a smooth projective variety admit Serre-invariant stability conditions. In the recent paper \cite{KP} the authors show that the Kuznetsov component (called residual category) of almost all Fano complete intersections of codimension $\geq 2$ does not admit Serre-invariant stability conditions.

In \cite[Theorem 1.1]{FP} a criterion is proved which ensures that a fractional Calabi--Yau category of dimension $\leq 2$ admits a unique Serre-invariant stability condition, up to
the action of $\widetilde{\mathrm{GL}}^+_2(\R)$. In Corollary \ref{cor_uniqueness} we show this criterion applies to the Kuznetsov component of a GM threefold. Note that this result was already known by \cite[Theorem 4.25]{JLLZ}. In particular, all known stability conditions on $\Ku(X)$ for $X$ a Fano threefold of Picard rank $1$, index $2$ or index $1$ and even genus $\geq 6$ are Serre-invariant. 

The next interesting question is to investigate whether the property of Serre-invariance characterize the stability conditions on $\Ku(X)$, providing a complete description of the stability manifold as in the case of curves \cite{Macri} (see Remark \ref{rmk_quest}). 

Stability conditions on the Kuznetsov component of a GM fourfold have been constructed in \cite{PPZ}. However, this existence is not shown through an explicit construction for special GM fourfolds, where it follows from the proof of the duality conjecture for GM varieties in \cite{KP_cones}. The stability conditions constructed in Theorem \ref{thm_application} are, to the authors' knowledge, the first explicit ones defined on special GM fourfolds. In the work in preparation \cite{PPZ_preparation}, Theorems \ref{thm_main} and \ref{thm_application} are useful to study properties (like non-emptyness) of moduli spaces of stable objects in $\Ku(X)$ of an ordinary GM threefold $X$, together with the results in \cite{PPZ} on moduli spaces on the associated special GM fourfold.  
\end{rw}

\begin{str}
Our proof of Theorem \ref{thm_main} is inspired by the approach used in \cite{PY} to address the corresponding question for cubic threefolds, although the situation in the case of GM threefolds is more complicated. Roughly speaking, the main issue is the presence of the rank two exceptional bundle $\UU_X$, which does not allow to use the same argument applied for cubic threefolds, where the exceptional objects were two line bundles. In fact, the inverse of the Serre functor decomposes as a composition of the left mutations with respect to $\OO_X$ and $\UU_X$, a twist by $\mathcal{O}_X (H)$ and a shift; in the case of cubic threefolds, we had a similar decomposition, but with only one left mutation with respect to the line bundle $\OO_X$. 

To overcome this problem, we first induce stability conditions $\sigma(s,q)$ on $\textrm{Ku}(X)$ from (a double tilt of) tilt stability conditions defined on $\D(X)$, via the criterion given in Proposition 5.1 of \cite{BLMS}, for pairs $(s,q)$ in a region that is slightly larger than the one considered in \cite{BLMS}, lying above the boundary given by Li in \cite{Li_Fano3} (see Propositions \ref{prop_Libound} and \ref{prop_stabcondKu}). Next, fixing one such stability condition, we study the action on it of the left mutation with respect to $\OO_X$ and then with respect to $\UU_X$. These left mutations induce equivalences between the right orthogonals of two distinct exceptional collections in $\D(X)$ and $\Ku(X)$. The Serre invariance of the considered stability condition follows from showing that its image via these equivalences is in the same $\widetilde{\mathrm{GL}}^+_2(\R)$-orbit of the induced stability conditions $\sigma(s,q)$. See Section \ref{sec_proof} for a more detailed summary of the argument.

The proof of Theorem \ref{thm_application} consists in showing that Serre-invariant stability conditions, for instance the induced stability conditions $\sigma(s,q)$, are invariant with respect to the $\mathbb{Z}/ 2\mathbb{Z}$-action on the Kuznetsov component of the GM threefold.
\end{str}

\begin{plan}
In Section \ref{sec_preliminary} we review the notions of (weak) stability conditions, GM varieties and their Kuznetsov components, and the construction of stability conditions in the case of GM threefolds. Section \ref{sec_actionSerrefunctor} is devoted to the proof of Theorem \ref{thm_main}. In Section \ref{sec_staboverLibound} we induce stability conditions on the Kuznetsov component of $X$ following the method in \cite{BLMS}, enlarging the region parametrizing them over Li's boundary, defined in \cite{Li_Fano3}. In Section \ref{sec_proof} we outline the proof of Theorem \ref{thm_main}, which will be performed in Sections \ref{sec_stabonKu3}, \ref{sec_stabonKu2}, \ref{sec_endofproof}. Section \ref{sec_applications} is devoted to the proof of Theorem \ref{thm_application} and Corollary \ref{cor_uniqueness}.  
\end{plan}

\begin{ack}
We would like to thank Xiaolei Zhao for many useful discussions and comments during the preparation of this work. We are grateful to Arend Bayer, Soheyla Feyzbakhsh, Chunyi Li, Emanuele Macrì, Paolo Stellari, Shizhuo Zhang for many interesting conversations. We wish to thank the referees for careful reading of the paper and for pointing out many inaccuracies.

L.P.\ is supported
by the national research project PRIN 2017 Moduli and Lie Theory. E.R.\ is partially supported by NSF FRG grant DMS-2052665.
\end{ack}

\section{Preliminaries on GM varieties and stability conditions} \label{sec_preliminary} 

In this section we review the definitions of (weak) stability conditions, Gushel--Mukai varieties and some basic properties of their Kuznetsov components. Then in the case of Gushel--Mukai threefolds, we recall the construction of stability conditions from \cite{BLMS}. We work over the field of complex numbers $\mathbb{C}$ throughout this paper.

\subsection{(Weak) stability conditions and tilting}
A (weak) stability condition on a triangulated category $\mathcal{T}$ is given by two pieces of data: a full subcategory $\mathcal{A} \subseteq \mathcal{T}$ called a heart of a bounded t-structure and a group homomorphism $Z: K(\mathcal{A}) \to \mathbb{C}$ called a (weak) stability function. We review these definitions now.
\begin{dfn}[\cite{Bridgeland07}, Lemma 3.2]
A \textit{heart of a bounded t-structure} is a full subcategory $\mathcal{A} \subseteq \mathcal{T}$ such that:
\begin{enumerate}
    \item For any $E,F \in \mathcal{A}$ and $k<0$, we have $\Hom(E,F [k])=0$.
    \item For any $E \in \mathcal{T}$, there is a filtration:
    \begin{align*}
        0=E_0 \xrightarrow{\phi_1} E_1 \xrightarrow{\phi_2} \cdot \cdot \cdot \xrightarrow{\phi_m} E_m = E
    \end{align*}
    such that for each $i$, $\textrm{Cone}(\phi_i) \cong A_i [k_i]$ for some $A_i \in \mathcal{A}$ and $k_1 >k_2 > \cdot \cdot \cdot > k_m$.
\end{enumerate}
\end{dfn}
\noindent A heart of a bounded t-structure is an abelian subcategory of $\mathcal{T}$ \cite{BBD}. We now define the (weak) stability functions mentioned above.
\begin{dfn}
Let $\mathcal{A}$ be an abelian category. A \textit{weak stability function} on $\mathcal{A}$ is a homomorphism of groups:
\begin{align*}
    Z \colon K(\mathcal{A}) &\to \mathbb{C} \\
    E &\mapsto \mathfrak{R}Z(E) + i \mathfrak{I}Z(E)
\end{align*}
where $K(\mathcal{A})$ denotes the Grothendieck group of $\mathcal{A}$, such that for all $0 \neq E \in \mathcal{A}$, we have $\mathfrak{I}Z(E) \geq 0$ and $\mathfrak{I}Z(E) = 0$ implies $\mathfrak{R}Z(E) \leq 0$. We say that $Z$ is a \textit{stability function} if, in addition, when $\mathfrak{I}Z(E) = 0$ we have $\mathfrak{R}Z(E) < 0$. 
\end{dfn}
We denote by $K(\TT)$ the Grothendieck group of $\TT$. Let $\Lambda$ be a finite rank lattice and $v \colon K(\mathcal{T}) \to \Lambda$ a surjective group homomorphism. 
\begin{dfn}
A \textit{weak stability condition} on $\mathcal{T}$ with respect to $\Lambda$ is a pair $\sigma = (\mathcal{A}, Z)$, where $\mathcal{A}$ is a heart of a bounded t-structure and $Z \colon \Lambda \to \mathbb{C}$ is a group homomorphism, such that:
\begin{enumerate}
    \item The composition $K(\mathcal{A})\cong K(\mathcal{T}) \xrightarrow{v} \Lambda \xrightarrow{Z} \mathbb{C}$ is a weak stability function on $\mathcal{A}$. We will omit the function $v$ and write $Z(E) = Z(v(E))$ for brevity. Given such a $Z$, we may define the slope of any $E \in \mathcal{A}$ as:
    \begin{align*}
       \mu_{\sigma}(E) = \begin{cases} 
      - \frac{\mathfrak{R}Z(E)}{\mathfrak{I}Z(E)} & \mathfrak{I}Z(E) \neq 0 \\
     +\infty & \textrm{otherwise.} 
   \end{cases} 
    \end{align*}
    We also obtain a notion of semistability (stability): we say $0 \neq E \in \mathcal{A}$ is $\sigma$-semistable ($\sigma$-stable) if for any nonzero, proper subobject $F \hookrightarrow E$, we have $\mu_{\sigma} (F) \leq \mu_{\sigma} (E)$ ($\mu_{\sigma} (F) < \mu_{\sigma} (E/F)$). 
    \item Any $E \in \mathcal{A}$ admits a Harder-Narasimhan filtration with $\sigma$-semistable factors. Explicitly, this means that given $E \in \mathcal{A}$, there is a filtration:
    \begin{align*}
         0=E_0 \xrightarrow{\phi_1} E_1 \xrightarrow{\phi_2} \cdot \cdot \cdot \xrightarrow{\phi_m} E_m = E
    \end{align*}
    such that $E_i/ E_{i-1}$ is $\sigma$-semistable, with $\mu_{\sigma} (E_1/E_0) > \cdot \cdot \cdot > \mu_{\sigma} (E_m/E_{m-1})$.
    \item (Support Property) There is a quadratic form $Q$ on $\Lambda \otimes \mathbb{R}$ such that $Q|_{\textrm{ker}Z}$ is negative-definite, and $Q(E) \geq 0$ for all $\sigma$-semistable $E \in \mathcal{A}$. 
\end{enumerate}
\end{dfn}
\begin{dfn}
A weak stability condition $\sigma = (\mathcal{A}, Z)$ on $\mathcal{T}$ with respect to $\Lambda$ is called a \textit{stability condition} if $Z$ is a stability function. 
\end{dfn}
Fix a (weak) stability condition $\sigma= (\mathcal{A}, Z)$ on $\mathcal{T}$. Given a semistable object $E \in \mathcal{A}$ with $Z(E) \neq 0$, we define the \textit{phase} of $E$ as
\begin{align*}
    \phi (E) = \frac{1}{\pi} \textrm{arg}(Z(E)).
\end{align*}
If $Z(E)=0$, we set $\phi(E) = 1$, and for any shift $E[n]$, we define $\phi(E[n])= \phi(E)+n$. The notion of phase of a semistable object $E \in \mathcal{A}$ naturally gives rise to a \textit{slicing} of $\mathcal{T}$.
\begin{dfn}
Let $\sigma$ be a (weak) stability condition on $\TT$. The \textit{slicing} of $\mathcal{T}$ associated to $\sigma$ is a collection $\mathcal{P}$ of full additive subcategories $\mathcal{P}(\phi)$ of $\mathcal{T}$ for each $\phi \in \mathbb{R}$ such that: 
\begin{enumerate}
    \item For $\phi \in (0,1]$, $\mathcal{P}(\phi)$ is the subcategory of all $\sigma$-semistable objects of phase $\phi$, together with the zero object.
    \item For $\phi \in (0,1]$ and $n \in \mathbb{Z}$, $\mathcal{P}(\phi +n)= \mathcal{P}(\phi)[n]$.
\end{enumerate}
\end{dfn}
We denote by $\PP(I)$ the extension-closed subcategory of $\TT$ generated by the subcategories $\PP(\phi)$ with $\phi \in I$, where $I \subset \R$ is an interval. Given a (weak) stability condition $\sigma$ with slicing $\mathcal{P}$, the heart $\mathcal{A}$ is recovered via $\mathcal{A}= \mathcal{P}((0,1])$, and conversely, a slicing $\mathcal{P}$ arises from $\sigma$ immediately by definition. In the next, we will use the notations $(\mathcal{A}, Z)$ and $(\mathcal{P},Z)$ for a (weak) stability condition interchangeably.

We write $\textrm{Stab}_{\Lambda}(\mathcal{T})$ to denote the set of stability conditions on $\mathcal{T}$. The space $\textrm{Stab}_{\Lambda}(\mathcal{T})$ can be given a metrizable topology in a natural way, and Bridgeland \cite{Bridgeland07} proved that with this topology, the map $\textrm{Stab}_{\Lambda} (\mathcal{T}) \to \textrm{Hom}(\Lambda, \mathbb{C})$ given by $(\mathcal{A}, Z) \mapsto Z$ is a local homeomorphism, hence $\textrm{Stab}(\mathcal{T})$ is a complex manifold of dimension $\textrm{rk}(\Lambda)$. 
\\
The manifold $\textrm{Stab}_{\Lambda}(\mathcal{T})$ admits two natural group actions, one from the universal cover of  $\textrm{GL}_2^+(\mathbb{R})$ (denoted $\widetilde{\textrm{GL}}_2^+(\mathbb{R})$) and one from the group $\textrm{Aut}_{\Lambda}(\mathcal{T})$ of exact autoequivalences which are compatible with $v$. For the former of these, given some $\widetilde{g}=(g,M) \in \widetilde{\textrm{GL}}_2^+(\mathbb{R})$ with $M \in \textrm{GL}_2^+(\mathbb{R})$ and $g: \mathbb{R} \to \mathbb{R}$ increasing with $g(\phi +1 )= g(\phi) +1$, the action on a stability condition $\sigma = (\mathcal{P}, Z)$ is given by $\sigma \cdot \widetilde{g} = (\mathcal{P}', M^{-1} \circ Z)$, where $\mathcal{P}'(\phi) = \mathcal{P}(g(\phi))$. For the $\textrm{Aut}_{\Lambda}(\mathcal{T})$-action, given some $\Phi \in \text{Aut}_{\Lambda}(\mathcal{T})$, we have $\Phi \cdot \sigma = (\Phi (\mathcal{P}), Z \circ \Phi_*^{-1})$, where $\Phi_*$ is the induced automorphism on $K(\mathcal{T})$.

The first issue in the construction of stability conditions is to produce a suitable heart of a bounded t-structure. For instance, the canonical choice of $\Coh(X)$ cannot be the heart of a stability condition with respect to the numerical Grothendieck group $\Lambda=\NN(X)$ of $X$, unless $X$ is a curve \cite{Toda}.
However, if we have a (weak) stability condition $(\mathcal{A}, Z)$, it is sometimes possible to produce a new heart by \textit{tilting} the old one. We discuss this procedure now.
\begin{dfn}
Let $\mathcal{A}$ be an abelian category. A \textit{torsion pair} is a pair of two full, additive subcategories $(\mathcal{F}, \mathcal{T})$ of $\AA$ such that:
\begin{enumerate}
    \item For any $T \in \mathcal{T}, F \in \mathcal{F}$, we have $\Hom( T, F) =0$.
    \item Given any $E \in \mathcal{A}$, there are $T \in \mathcal{T}, F \in \mathcal{F}$ and a short exact sequence:
    \begin{align*}
        0 \to T \to E \to F \to 0
    \end{align*}
\end{enumerate}
\end{dfn}
\noindent The importance of this notion comes from the following theorem:
\begin{thm}[\cite{HRS}]  
Let $\mathcal{A} \subset \D(X)$ be a heart of a bounded t-structure and let $(\mathcal{F}, \mathcal{T})$ be a torsion pair in $\mathcal{A}$. Then the extension-closure $\langle \mathcal{F}[1], \mathcal{T}\rangle$ is also a heart of a bounded t-structure in $\D(X)$. 
\end{thm}
Given a (weak) stability condition $\sigma = (\mathcal{A}, Z)$ on $\D(X)$, one may produce a new heart according to the theorem above by choosing any $\mu \in \mathbb{R}$ and considering the following torsion pair:
\begin{align*}
    \mathcal{F}_{\sigma}^{\mu} &= \langle E \in \mathcal{A}: E \textrm{ is semistable with } \mu_{\sigma} (E) \leq \mu \rangle  \\
    \mathcal{T}_{\sigma}^{\mu} &= \langle E \in \mathcal{A}: E \textrm{ is semistable with } \mu_{\sigma} (E) > \mu \rangle .
\end{align*}
We say that the new heart $\langle \mathcal{F}_{\sigma}^{\mu}[1], \mathcal{T}_{\sigma}^{\mu} \rangle$ is constructed by tilting the (weak) stability condition $\sigma$ at the slope $\mu$. This construction is ubiquitous in what follows.
\begin{example}(\cite[Example 2.8]{BLMS})  \label{ex_tiltslopestab}
Let $X$ be a smooth projective variety of dimension $n$ with an ample class $H$. We have that the group morphism 
$$Z_H \colon \Lambda \cong \Z^2 \to \C; \quad (H^n\rk(E), H^{n-1}\ch_1(E)) \mapsto -H^{n-1}\ch_1(E)+ H^n\rk(E) \sqrt{-1}$$ 
defines a weak stability function on $\Coh(X)$. Moreover, the pair $\sigma_H=(\Coh(X), Z_H)$ is a weak stability condition on $\D(X)$ with respect to $\Lambda$, known as \emph{slope stability}. The slope with respect to $\sigma_H$ is denoted by $\mu_H$. Furthermore, if $n=1$, then $\sigma_H$ is a stability condition on $\D(X)$.

We remark that slope semistable coherent sheaves satisfy the classical Bogomolov--Gieseker inequality: for every $\mu_H$-semistable $E \in \Coh(X)$ we have the inequality
\begin{equation}
\label{eq_BGineq}    
(H^{n-1}\ch_1(E))^2 -2H^n\rk(E) H^{n-2}\ch_2(E) \geq 0
\end{equation}
\end{example}

\subsection{GM varieties and Kuznetsov components} \label{sec_GM}
A Gushel--Mukai (GM) variety of dimension $n$, for $2 \leq n \leq 6$, is a smooth intersection $$\textrm{Cone}(\textrm{Gr}(2,5)) \cap Q,$$ where $\textrm{Cone}(\textrm{Gr}(2,5))$ is the projective cone over the Pl\"ucker-embedded Grassmanian $\textrm{Gr}(2,5) \hookrightarrow \mathbb{P}^9$ and $Q$ is a quadric hypersurface in some $\mathbb{P}(W) \cong \mathbb{P}^{n+4} \hookrightarrow \mathbb{P}^{10}$. Gushel \cite{Gushel} and Mukai \cite{Mukai} showed that for $n \geq 3$, GM varieties are precisely the Fano varieties of Picard number $1$, degree $10$ and coindex $3$, while if $n=2$, GM surfaces are Brill--Noether general polarized K3 surfaces. If the vertex of the cone $\textrm{Cone}(\textrm{Gr}(2,5))$ is not in the linear section $\P(W)$, then $X$ is an \emph{ordinary} GM variety, otherwise $X$ is a \emph{special} GM variety.  

Kuznetsov and Perry \cite{KP} proved that the bounded derived category $\D(X)$ of a GM variety $X$ of dimension $n \geq 3$ admits a semiorthogonal decomposition of the form
\begin{equation} \label{eq_sod}
    \D(X) = \langle \textrm{Ku}(X),\mathcal{O}_X, \mathcal{U}^\vee_X, ..., \mathcal{O}_X ((n-3)H), \mathcal{U}^\vee_X ((n-3)H) \rangle,
\end{equation}
where $\mathcal{U}_X$ is the pullback to $X$ of the rank $2$ tautological subbundle on the Grassmannian, $H \subset X$ is a hyperplane class and $\textrm{Ku}(X) :=\langle \mathcal{O}_X, \mathcal{U}_X^\vee, ..., \mathcal{O}_X ((n-3)H), \mathcal{U}_X^\vee ((n-3)H) \rangle^{\perp}$ is the \textit{Kuznetsov component}. For $n=2$, set $\Ku(X):= \D(X)$.

Since $\Ku(X)$ is an admissible subcategory of $\D(X)$, it admits a Serre functor, which we denote by $S_{\Ku(X)}$. By \cite[Proposition 2.6]{KP} (which makes use of \cite[Corollaries 3.7, 3.8]{Kuz_calabi}) the Serre functor of $\Ku(X)$ has the following property:
\begin{itemize}
\item if $n$ is even, then $S_{\Ku(X)} \cong [2]$;
\item if $n$ is odd, then $S_{\Ku(X)} \cong \sigma[2]$ for a
nontrivial involutive autoequivalence $\sigma$ of $\Ku(X)$.
\end{itemize}
Moreover, computing the Hochschild homology \cite[Proposition 2.9]{KP} one sees that if $n$ is even, then $\Ku(X)$ is a noncommutative K3 surface, while for $n$ odd $\Ku(X)$ is a noncommutative Enriques surface.

Let $X$ be a GM threefold. Since $\omega_X \cong \OO_X(-H)$, by Serre duality we can write the semiorthogonal decomposition \eqref{eq_sod} as
$$\D(X)=\langle \Ku(X), \OO_X, \UU_X^\vee \rangle=\langle \UU_X^\vee(-H), \Ku(X), \OO_X \rangle.$$
Since $\UU_X^\vee(-H) \cong \UU_X$, we obtain the alternative semiorthogonal decomposition 
$$\D(X)=\langle \L_{\UU_X}(\Ku(X)), \UU_X, \OO_X \rangle$$ 
which is the one used in \cite{BLMS} for the construction of stability conditions. Note that $\Ku(X)$ and $\L_{\UU_X}(\Ku(X))$ are equivalent by \cite[Proposition 3.8]{Kuz_cubic}, \cite{Bondal}. In order to be compatible with  \cite{BLMS}, we set
\begin{equation}
\label{eq_defKu}    
\Ku(X)_1:= \langle \UU_X, \OO_X \rangle^{\perp}
\end{equation}
sitting in
$$\D(X)=\langle \Ku(X)_1, \UU_X, \OO_X \rangle$$
(in fact, in the rest of this paper we will need to be precise on which Kuznetsov component we are working on, see Section \ref{sec_staboverLibound}).  By \cite[Proposition 3.9]{Kuz}, the numerical Grothendieck group $\NN(\Ku(X)_1)$ of $\textrm{Ku}(X)_1$ satisfies $\NN(\Ku(X)_1) \cong \mathbb{Z}^{\oplus 2}$ and a basis is
\begin{align} \label{eq_basis}
    b_1 &= 1- \frac{3}{10}H^2 + \frac{1}{20}H^3 \\
    b_2 &= H- \frac{3}{5}H^2 + \frac{1}{60}H^3. \nonumber
\end{align}
The Todd class of $X$ is
\begin{equation*}
\td(X)=1 + \frac{1}{2}H + \frac{17}{60}H^2 + \frac{1}{10}H^3.    
\end{equation*}

%The Kuznetsov component $\textrm{Ku}(X)$ is a triangulated subcategory of $D^b (X)$ admitting a Serre functor $S_{K}$. It is known that $\textrm{Ku}(X)$ admits Bridgeland stability conditions for any $X$, this was shown in [] and we will additionally produce them later using a method of Bayer, Lahoz, Macri and Stellari. Here, we wish to show that in the case $\dim X =3$, there exists a Bridgeland stability condition $\sigma$ on $\textrm{Ku}(X)$ whose orbit under the natural $\widetilde{\textrm{GL}}(2, \mathbb{R})$-action is preserved by $S_{K}$. 
%\\
%\\
%The Serre functor $S_K$ is described by Kuzentsov in [] as $S_K^{-1}= \mathbb{L} \circ S^{-1}$, where $S: D^b X \to D^b X$ is the ordinary Serre functor, i.e. $S(-)= (-) \otimes \omega_X [3]$ and $\mathbb{L}$ is left mutation over $\langle \mathcal{O}_X, \mathcal{U}^\vee \rangle$. 
%We utilize an approach analogous to the one used in the work of Pertusi and Yang [] for the proof of the same result for cubic threefolds. Specifically, we induce Bridgeland stability conditions $\sigma (\alpha, \beta)$ on $\textrm{Ku}(X)$ via weak stability conditions on $D^b (X)$ using Prop. 2.8 in [], the construction of the weak stability conditions is similar. Subsequently, we fix an orbit $K$ of one such stability condition, and showing that $S_K \cdot K=K$ effectively breaks down to showing that $\mathbb{L} \circ S^{-1}$ simply tilts the hearts of bounded t-structures of the stability conditions lying in $K$. 

\subsection{Stability conditions on $\Ku(X)$}
\label{sec_stabcondonKu}
The existence of stability conditions on $\Ku(X)$ of a GM variety is known. More precisely, if $X$ has dimension $2$, this follows from Bridgeland's work \cite{Bridgeland}. By the duality conjecture \cite[Theorem 1.6]{KP_cones} if $X$ has dimension $6$ or $5$, the problem reduces to the same question in dimension $4$ and $3$, respectively: if $X$ is a GM fourfold, this is proved in \cite{PPZ}, while the case of GM threefolds is solved by Bayer, Lahoz, Macr\`i and Stellari in \cite{BLMS}.

In this section we focus on GM threefolds and we review the construction of stability conditions on $\Ku(X)$ defined in \eqref{eq_defKu} given in \cite{BLMS}.
%where $\mathcal{E}$ is a slope-stable, rank 2 vector bundle on $X$ with $c_1 (\mathcal{E})= -H$ and $\textrm{ch}_2 (\mathcal{E})= L$, where $L$ is the class of a line on $X$. The existence of such $\mathcal{E}$ is proven in []. 

%We recall briefly the construction of the weak stability conditions $\sigma_{\alpha, \beta}^{\mu}$ given in [] and the criteria which guarantee that they induce (strong) stability conditions on $\textrm{Ku}(X)$. 

Stability conditions on $\Ku(X)$ are induced from double-tilted slope stability on $\D(X)$.
First for $\alpha >0$ and $\beta \in \mathbb{R}$, consider the weak stability conditions on $\D(X)$ of the form $$\sigma_{\alpha, \beta} = (\textrm{Coh}^{\beta}(X), Z_{\alpha, \beta})$$  
with respect to the rank-$3$ lattice $\Lambda$ generated by vectors $(H^3 \rk(E), H^2\ch_1(E), H\ch_2(E))$ for $E \in \D(X)$. Here, $\textrm{Coh}^{\beta} (X)$ is the heart of a bounded t-structure obtained by tilting $\textrm{Coh}(X)$ with respect to slope stability at slope $\mu = \beta$ (see Example \ref{ex_tiltslopestab}), and the central charge $Z_{\alpha, \beta}$ is 
\begin{align}
    Z_{\alpha, \beta}(E) = \frac{1}{2} \alpha^2 H^3 \cdot \textrm{ch}^{\beta}_0 (E) - H \cdot \textrm{ch}^{\beta}_2(E) +i H^2 \cdot \textrm{ch}^{\beta}_1 (E),
\end{align}
where $\textrm{ch}^{\beta}_i (E)$ is the $i$-th component of the twisted Chern character $\textrm{ch}^{\beta} (-) = e^{-\beta H} \cdot \textrm{ch}(-)$ (see \cite[Proposition 2.12]{BLMS}).
For $E \in \Coh^\beta(X)$ the slope of $E$ defined by $\sigma_{\alpha, \beta}$ is
\begin{equation} \label{eq_mualphabeta}
 \mu_{\alpha, \beta}(E)=
\begin{cases}
-\frac{\Re Z_{\alpha, \beta}(E)}{\Im Z_{\alpha, \beta}(E)} & \text{if } \Im Z_{\alpha, \beta}(E) > 0\\
+ \infty & \text{otherwise}.
\end{cases}   
\end{equation}
Note that $\sigma_{\alpha, \beta}$-semistable objects satisfy the inequality \eqref{eq_BGineq} which can be taken as the quadratic form satisfying the support property.

Second for $\mu \in \R$, denote by $\textrm{Coh}_{\alpha, \beta}^{\mu} (X)$ the heart obtained by tilting $\Coh^\beta(X)$ with respect to $\sigma_{\alpha, \beta}$ at slope $\mu_{\alpha, \beta}= \mu$. Fix $u \in \mathbb{C}$ such that $u$ is the unit vector in the upper half plane with $\mu = - \frac{\textrm{Re}(u)}{\textrm{Im}(u)}$. By \cite[Proposition 2.15]{BLMS} we have that 
\begin{equation} \label{eq_weakstabcond} 
\sigma_{\alpha, \beta}^{\mu} = (\textrm{Coh}^{\mu}_{\alpha, \beta} (X), Z^{\mu}_{\alpha, \beta})    
\end{equation}
is a weak stability condition on $\D(X)$ with respect to $\Lambda$, where $Z^{\mu}_{\alpha, \beta} = \frac{1}{u}Z_{\alpha, \beta}$. 

Now we recall the following criterion from \cite{BLMS}, which is useful for determining when weak stability conditions defined on $\D(X)$, like those above, restrict to stability conditions on the orthogonal complement of a subcategory determined by an exceptional collection. In the following, $\TT$ is a triangulated category with Serre functor $S$, $E_0, \cdot \cdot \cdot, E_m$ are exceptional objects in $\TT$ and $\DD= \langle E_0, \cdot \cdot \cdot, E_m \rangle$, giving a semiorthogonal decomposition $\TT= \langle \DD^{\perp}, \DD \rangle$. 

\begin{prop}[\cite{BLMS}, Proposition 5.1]
\label{prop_criterion}
Let $\sigma =(\A, Z)$ be a weak stability condition on $\T$. Assume that:
\begin{enumerate}
    \item $E_i \in \A$
    \item $S(E_i) \in \A[1]$
    \item $Z(E_i) \neq 0$ for all $i$.
\end{enumerate}
If for all $0 \neq E \in \A \cap \DD^{\perp}= \A_1$ we have $Z(E) \neq 0$, then the pair $(\A_1, Z|_{\A_1})$ defines a stability condition on $\DD^{\perp}$.  
\end{prop}

The criterion above was applied in \cite{BLMS} to show the following existence result.
\begin{thm}[\cite{BLMS}, Theorem 6.9] \label{thm_BLMSresult}
Let $X$ be a GM threefold. Then the weak stability conditions $\sigma^{\mu}_{\alpha, \beta}$ defined in \eqref{eq_weakstabcond} induce stability conditions on $\emph{Ku}(X)_1$ so long as $\alpha >0$ is sufficiently close to $0$, $\beta>-1$ is sufficiently close to $-1$ and $\mu_{\alpha, \beta} (\mathcal{O}_X (-1)[1])< \mu < \mu_{\alpha, \beta} (\mathcal{U}_X)$.  
\end{thm}

\section{Action of the Serre functor on stability conditions on $\Ku(X)$} \label{sec_actionSerrefunctor}
This section is devoted to the proof of Theorem \ref{thm_main}. In Section \ref{sec_staboverLibound} we induce stability conditions on the Kuznetsov component of $X$ from (a tilt of the) tilt stability conditions lying over Li's boundary, defined in \cite{Li_Fano3}. This allows to enlarge the region where there are induced stability conditions on $\Ku(X)$ with the method of \cite{BLMS} and will be useful in Section \ref{sec_stabonKu2}. In Section \ref{sec_proof} we outline the proof of Theorem \ref{thm_main}, which will be carried out in Sections \ref{sec_stabonKu3}, \ref{sec_stabonKu2}, \ref{sec_endofproof}. 

\subsection{Stability conditions over Li's boundary} \label{sec_staboverLibound}
Let $X$ be a GM threefold. Note that we have the following semiorthogonal decompositions: 
\begin{gather}
\D(X)= \langle \Ku(X)_1, \UU_X, \OO_X \rangle, \label{eq_Ku1} \\
\D(X)= \langle \Ku(X)_2, \OO_X, \UU_X^\vee \rangle, \label{eq_Ku2} \\
\D(X)= \langle \Ku(X)_3, \UU_X^\vee, \OO_X(H) \rangle \label{eq_Ku3}    
\end{gather} 
Here $\Ku(X)_1$ was already defined in \eqref{eq_defKu} and $\Ku(X)_2:=\Ku(X)$ as in \eqref{eq_sod}. We can obtain \eqref{eq_Ku3} tensoring \eqref{eq_Ku1} by $\OO_X(H)$ and setting $\Ku(X)_3:= \Ku(X)_1(H)$. Analogously, by Serre duality we have
$$\D(X)= \langle \Ku(X)_3, \UU_X^\vee, \OO_X(H) \rangle= \langle \OO_X, \Ku(X)_3, \UU_X^\vee \rangle=\langle \L_{\OO_X}(\Ku(X)_3), \OO_X, \UU_X^\vee\rangle$$
so we get $\Ku(X)_2=\L_{\OO_X}(\Ku(X)_3)$. Note also that $\Ku(X)_1$, $\Ku(X)_2$, $\Ku(X)_3$ are equivalent to each others by \cite[Proposition 3.8]{Kuz_cubic}, \cite{Bondal}. 

As in \cite[Section 1]{LiZhao_birational}, \cite{Li_Fano3}, we consider the following reparametrization of the tilt stability condition $\sigma_{\alpha,\beta}$, whose definition is recalled in Section \ref{sec_preliminary}. For $q>0, s \in \R$ and $E \in \D(X)$ we define
$$Z_{s,q}(E)= -(H \cdot \ch_2(E)-q \rk(E)H^3)+ \sqrt{-1}(H^2 \cdot \ch_1(E)-s\rk(E)H^3).$$
For $E \in \Coh^s(X)$ we have the associated slope function
$$\mu_{s,q}(E)= \frac{H \cdot \ch_2(E)-q \rk(E)H^3}{H^2 \cdot \ch_1(E)-s\rk(E)H^3}.$$
Then for $\alpha>0, \beta \in \R$, setting $s=\beta$, $q=\frac{\alpha^2+\beta^2}{2}$, it follows that
\begin{equation} \label{eq_reparametrizedslope}
\mu_{\alpha,\beta}= \mu_{s,q}-s,   
\end{equation}
where $\mu_{\alpha, \beta}$ is defined in \eqref{eq_mualphabeta} and $\mu_{s,q}$ is defined in the above formula, and for $q> \frac{1}{2}s^2$ the pair $\sigma_{s,q}=(\Coh^s(X), Z_{s,q})$ defines a weak stability condition on $\D(X)$. 

For $E \in \D(X)$ we consider the reduced character
$$\widetilde{v}_H(E):=[H^3\rk(E):H^2 \cdot \ch_1(E): H \cdot \ch_2(E)]$$
which defines a point in a projective plane $\mathbb{P}^2_{\R}$ when $\widetilde{v}_H(E) \neq \textbf{0}$. If $\rk(E) \neq 0$, we consider the affine coordinates
$$\left ( s(E):=\frac{H^2 \cdot \ch_1(E)}{H^3\rk(E)},\, q(E):= \frac{H \cdot \ch_2(E)}{H^3\rk(E)} \right ) \in \mathbb{A}^2_{\R}.$$
Note that since the inequality \eqref{eq_BGineq} holds for $\sigma_{s,q}$-semistable objects, we have that points below the parabola $q= \frac{1}{2}s^2$ correspond to $\sigma_{s,q}$-semistable objects. Furthermore, the slope of a $\sigma_{s,q}$-semistable objects $E \in \Coh^s(X)$ is the gradient of the line connecting $(s,q)$ and $(s(E), q(E))$ (see Figure \eqref{fig_0}).

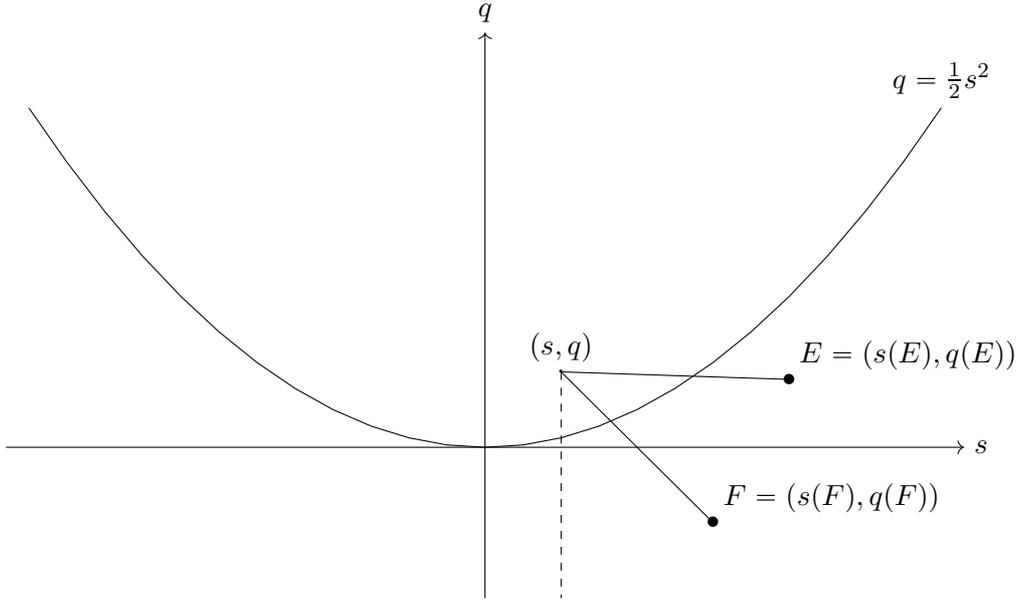
\begin{figure}[htb]
\centering
\begin{tikzpicture}[domain=-6:6]
\draw[->] (-6.3,0) -- (6.3,0) node[right] {$s$};
\draw[->] (0,-2) -- (0, 5.5) node[above] {$q$};
\draw plot (\x,{0.125*\x*\x}) node[above] {$q = \frac{1}{2}s^{2}$};
\coordinate (O) at (0,0);

\coordinate (E) at (4, 0.9);
\node  at (E) {$\bullet$};
\draw (E) node [above right]  {$E=(s(E),q(E))$}; 

\coordinate (F) at (3, -1);
\node  at (F) {$\bullet$};
\draw (F) node [above right]  {$F=(s(F),q(F))$}; 

\node  at (1,1) {$\cdot$};
\draw (1,1) node [above]  {$(s,q)$}; 

\draw (E) -- (1,1);
\draw (F) -- (1,1);
\draw[dashed] (1,1) -- (1, -2);
 
\end{tikzpicture}

\caption{If $E$, $F$ are $\sigma_{s,q}$-semistable, their $\mu_{s,q}$-slope is the gradient of the line connecting the point $(s,q)$ with $(s(E),q(E))$ and $(s(F),q(F))$, respectively. We may also compare the $\mu_{s,q}$-slope of $E$ and $F$, using the picture: $E$ has larger slope than $F$ is and only if the line connecting $E$ to $(s,q)$ is above the line connecting $F$ with $(s,q)$ (see \cite[Lemma 2]{LiZhao2}). \label{fig_0}}   
\end{figure}

By \cite[Theorem 0.3]{Li_Fano3} slope stable coherent sheaves on $X$ satisfy a stronger Bogomolov inequality. More precisely, in the affine plane $\mathbb{A}^2_{\R}$ we consider the open region
\begin{equation}
\label{eq_regionR}
R_{\frac{3}{20}}
\end{equation}
defined in \cite[Definition 3.1]{Li_Fano3} as the set of points above the curve $s^2-2q=\frac{3}{20}$ and above the tangent lines to the curve $s^2-2q=0$ at $\widetilde{v}_H(\OO_X(kH))$ for all $k \in \Z$ (see Figure \eqref{fig_1}). 

\begin{figure}[htb]
\centering
\begin{tikzpicture}[domain=-6:6]
\draw[->] (-6.3,0) -- (6.3,0) node[right] {$s$};
\draw[->] (0,-2) -- (0, 5.5) node[above] {$q$};
\draw plot (\x,{0.125*\x*\x}) node[above] {$q = \frac{1}{2}s^{2}$};
\coordinate (O) at (0,0);
\node  at (O) {$\bullet$};
\draw (O) node [above right]  {$\mathcal{O}_{X}$};

\draw plot (\x,{0.125*\x*\x -0.3}) node[right] {$q = \frac{1}{2}s^{2}- \frac{3}{40}$};

\coordinate (OH) at (-4,2);
\node  at (OH) {$\bullet$};
\draw (-4,2) node [above right]  {$\mathcal{O}_{X}(-H)$};

\coordinate (OK) at (4,2);
\node  at (OK) {$\bullet$};
\draw (4,2) node [above left]  {$\mathcal{O}_{X}(H)$};

\coordinate (P) at (-1.7,-0.2);
\draw (OH)--(P);

\coordinate (Q) at (1.7,-0.2);
\draw (OK)--(Q);

\coordinate (M) at (-2.3,0.38);
\node at (M) {$\bullet$}; 

\coordinate (N) at (2.3,0.38);
\node at (N) {$\bullet$}; 

\draw[red] (OH)--(M) ;
\draw[red] (OK)--(N) ;

\coordinate (A) at (-1.5,0);
\node at (A) {$\bullet$}; 

\coordinate (B) at (1.5,0);
\node at (B) {$\bullet$}; 
\draw[red] (A)--(B) ;

\draw[domain=-2.3:-1.5,color=red] plot (\x,{0.125*\x*\x -0.3});
\draw[domain=1.5:2.3,color=red] plot (\x,{0.125*\x*\x -0.3});

\end{tikzpicture}

\caption{We represent the boundary of the region $R_{\frac{3}{20}}$ among $\widetilde{v}_H(\OO_X(-H))$ and $\widetilde{v}_H(\OO_X(H))$ in red.  \label{fig_1}}  
\end{figure}

As a consequence, we obtain the following refined result.

\begin{prop}[\cite{BLMS}, Proposition 2.12, \cite{Li_Fano3}, Theorem 0.3]  \label{prop_Libound}
For $(s,q) \in R_{\frac{3}{20}}$, the pair $\sigma_{s,q}=(\emph{Coh}^s(X), Z_{s,q})$    defines a weak stability condition on $\D(X)$ with respect to the lattice $\Lambda^2_H \cong \Z^{\oplus 3}$ generated by the reduced Chern character. 
\end{prop}

Now using the same strategy as in \cite{BLMS} we can induce stability conditions on the Kuznetsov components \eqref{eq_Ku1},\eqref{eq_Ku2},\eqref{eq_Ku3} from $\sigma_{s,q}$ for certain values of $(s,q) \in R_{\frac{3}{20}}$. As done in \eqref{eq_weakstabcond}, we need to tilt a second time. For $\mu \in \R$, we  denote by $\Coh_{s,q}^\mu(X)$ the heart obtained by tilting $\Coh^s(X)$ with respect to $\sigma_{s,q}$ at $\mu$. Then \cite[Proposition 2.15]{BLMS}, which applies in the same way to the reparametrized tilt stability conditions, implies that $\sigma^{\mu}_{s,q}=(\Coh_{s,q}^\mu(X), Z^\mu_{s,q})$ is a weak stability condition on $\D(X).$ 

For $i=1,2,3$, we set 
$$\A(s,q):=\Coh^\mu_{s,q}(X) \cap \Ku(X)_i$$
and 
\begin{equation} \label{eq_defodZsq}
Z(s,q):=Z^\mu_{s,q}|_{\Ku(X)_i}    
\end{equation}
where $Z^\mu_{s,q}=\frac{1}{u}Z_{s,q}$ and $\mu=-\frac{\Re u}{\Im u}$. 
We also note that the exceptional bundles in the semiortoghonal decompositions \eqref{eq_Ku1}, \eqref{eq_Ku2}, \eqref{eq_Ku3} are on the boundary of $R_{\frac{3}{20}}$ as
$$\ch_{\leq 2}(\OO_X(kH))=(1,kH,\frac{k^2}{2}H^2), \quad \ch_{\leq 2}(\UU_X^\vee)=(2, H, \frac{1}{10}H^2),$$ $$\ch_{\leq 2}(\UU_X)=(2, -H, \frac{1}{10}H^2), \quad  \ch_{\leq 2}(\UU_X(-H))=(2, -3H, \frac{21}{10}H^2).$$

\begin{prop}
\label{prop_stabcondKu}
Let $(s,q)$ be points in the region $R_{\frac{3}{20}}$. 
\begin{enumerate}
    \item If $(s,q)$ is below the segment connecting $\widetilde{v}_H(\OO_X(-H))$ and $\widetilde{v}_H(\UU_X)$, then the pair $\sigma(s,q)=(\A(s,q), Z(s,q))$ defines a Bridgeland stability condition on $\Ku(X)_1$ with respect to $\Lambda^2_H$ for $\mu \in \R$ satisfying $\mu_{s,q}(\OO_X(-H)[1]) \leq \mu < \mu_{s,q}(\UU_X)$.
    \item If $(s,q)$ is below the segment connecting $\widetilde{v}_H(\OO_X)$ and $\widetilde{v}_H(\UU_X)$, then the pair $\sigma(s,q)=(\A(s,q), Z(s,q))$ defines a Bridgeland stability condition on $\Ku(X)_2$ with respect to $\Lambda^2_H$ for $\mu \in \R$ satisfying $\mu_{s,q}(\UU_X[1]) \leq \mu < \mu_{s,q}(\OO_X)$.
    \item If $(s,q)$ is below the segment connecting $\widetilde{v}_H(\OO_X)$ and $\widetilde{v}_H(\UU_X^\vee)$, then the pair $\sigma(s,q)=(\A(s,q), Z(s,q))$ defines a Bridgeland stability condition on $\Ku(X)_3$ with respect to $\Lambda^2_H$ for $\mu \in \R$ satisfying $\mu_{s,q}(\OO_X[1]) \leq \mu < \mu_{s,q}(\UU_X^\vee)$.
\end{enumerate}
\end{prop}
In Figure \ref{fig_2} we represent the regions where there are induced stability conditions as in Proposition \ref{prop_stabcondKu}.

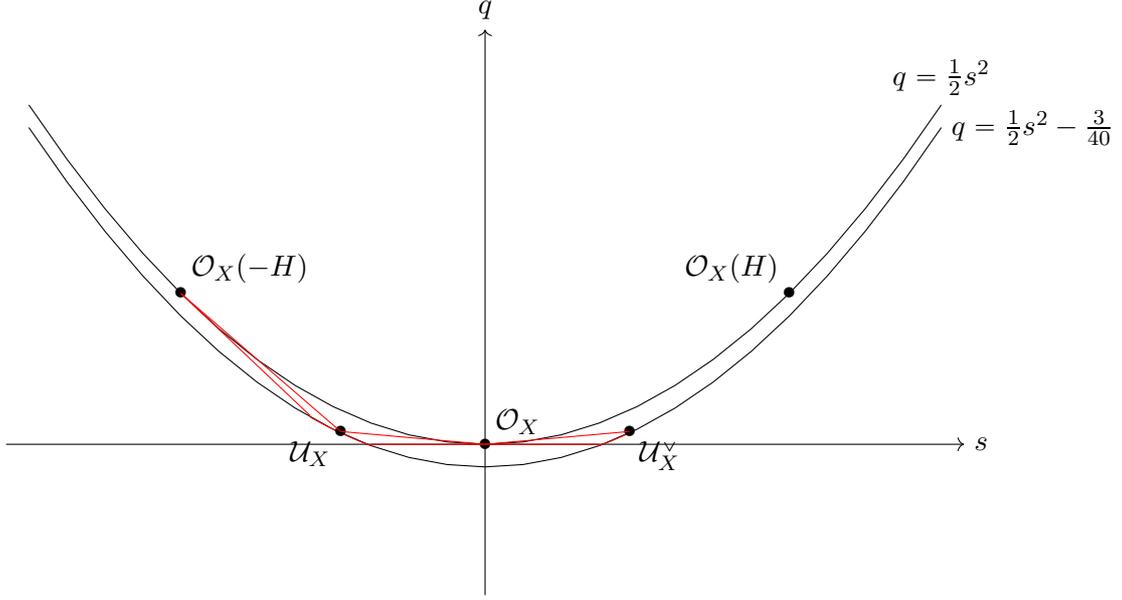
\begin{figure}[htb]
\centering
\begin{tikzpicture}[domain=-6:6]
\draw[->] (-6.3,0) -- (6.3,0) node[right] {$s$};
\draw[->] (0,-2) -- (0, 5.5) node[above] {$q$};
\draw plot (\x,{0.125*\x*\x}) node[above] {$q = \frac{1}{2}s^{2}$};
\coordinate (O) at (0,0);
\node  at (O) {$\bullet$};
\draw (O) node [above right]  {$\mathcal{O}_{X}$};

\draw plot (\x,{0.125*\x*\x -0.3}) node[right] {$q = \frac{1}{2}s^{2}- \frac{3}{40}$};

\coordinate (OH) at (-4,2);
\node  at (OH) {$\bullet$};
\draw (-4,2) node [above right]  {$\mathcal{O}_{X}(-H)$};

\coordinate (OK) at (4,2);
\node  at (OK) {$\bullet$};
\draw (4,2) node [above left]  {$\mathcal{O}_{X}(H)$};

\coordinate (M) at (-2.3,0.38);
%\node at (M) {$\bullet$}; 

\coordinate (N) at (2.3,0.38);
%\node at (N) {$\bullet$}; 

\draw[red] (OH)--(M) ;

\coordinate (A) at (-1.5,0);
%\node at (A) {$\bullet$}; 

\coordinate (B) at (1.5,0);
%\node at (B) {$\bullet$}; 
\draw[red] (A)--(B) ;

\draw[domain=-2.3:-1.5,color=red] plot (\x,{0.125*\x*\x -0.3});
\draw[domain=1.5:1.9,color=red] plot (\x,{0.125*\x*\x -0.3});

\coordinate (U) at (-1.9,0.17); 
\node at (U) {$\bullet$};  
\draw (U) node [below left]  {$\mathcal{U}_{X}$}; 

\coordinate (UV) at (1.9,0.17); 
\node at (UV) {$\bullet$};  
\draw (UV) node [below right]  {$\mathcal{U}_{X}^\vee$}; 
  
\draw[red] (OH)--(U); 
\draw[red] (U)--(O); 
\draw[red] (UV)--(O);

\end{tikzpicture}

\caption{We represent in red the boundary of the regions defined in Proposition \ref{prop_stabcondKu}. \label{fig_2}}  
\end{figure}

\begin{proof}
This is a refinement of \cite[Theorem 6.8]{BLMS}, where the statement is proved in the case of $\Ku(X)_1$ for $(s,q)$ above the parabola $q-\frac{1}{2}s^2=0$ and $\mu$ as in item 1. 

We study the case of $\Ku(X)_1$, the others can be treated analogously. Note that $\UU_X$, $\OO_X$, $\UU_X(-H)$, $\OO_X(-H)$ are slope stable sheaves with slope $-\frac{1}{2}$, $0$, $-\frac{3}{2}$, $-1$, respectively. Thus $\UU_X$, $\OO_X$, $\UU_X(-H)[1]$, $\OO_X(-H)[1]$ belong to $\Coh^s(X)$ for $-1 \leq s < -\frac{1}{2}$. Since these objects are on the boundary of $R_{\frac{3}{20}}$, by \cite[Corollary 3.11]{BMS} we have that $\UU_X$, $\OO_X$, $\UU_X(-H)[1]$, $\OO_X(-H)[1]$ are $\sigma_{s,q}$-stable in $\Coh^s(X)$. For $(s,q)$ as in the assumptions of item 1, by a direct computation or comparing the slopes using the picture, we see that
$$\mu_{s,q}(\UU_X(-H)[1]) < \mu_{s,q}(\OO_X(-H)[1]) < \mu_{s,q}(\UU_X) < \mu_{s,q}(\OO_X).$$
Thus for $\mu$ as in the statement, we have $\UU_X$, $\OO_X$, $\UU_X(-H)[2]$, $\OO_X(-H)[2]$ in $\Coh^\mu_{s,q}(X)$. Finally, by \cite[Lemma 2.16]{BLMS} objects in $\Coh^\mu_{s,q}(X)$ with vanishing central charge $Z^{\mu}_{s,q}$ are objects in $\Coh^s(X)$ with vanishing central charge $Z_{s,q}$, which are torsion sheaves supported on points. Since for such a sheaf $T$ we always have $\Hom(\OO_X, T) \neq 0$, we conclude that $T$ does not belong to $\Ku(X)_1$. Then Proposition \ref{prop_criterion} implies the statement.
\end{proof}

Note that we omit $\mu$ from the notation of the induced stability condition. In fact, $\sigma(s,q)$ does not depend on $\mu$, up to the action of $\widetilde{\text{GL}}^+(2,\R)$, as we show in the next lemma.

\begin{lemma} \label{lem_notdependonmu} 
Fix $i=1,2,3$. Let $(s,q)$ be a point in $R_{\frac{3}{20}}$, $\mu>\mu' \in \R$ satisfying the conditions in item (i) of Proposition \ref{prop_stabcondKu}. Then the stability condition induced from $\sigma^\mu_{s,q}$ is the same as the one induced from $\sigma^{\mu'}_{s,q}$, up to the $\widetilde{\emph{GL}}^+(2,\R)$-action.
\end{lemma}
\begin{proof}
Denote by $\sigma(s,q,\mu)$ and $\sigma(s,q,\mu')$ the induced stability conditions on $\Ku(X)_i$ corresponding to the choice of $\mu$ and $\mu'$, respectively. We claim that
$$\Coh^\mu_{s,q}(X) \subset \langle \Coh^{\mu'}_{s,q}(X), \Coh^{\mu'}_{s,q}(X)[1] \rangle.$$
Indeed, consider $F \in \Coh^s(X)$ semistable with $\mu_{s,q}(F)> \mu$, which is an object in $\Coh^{\mu}_{s,q}(X)$. Then $\mu_{s,q}(F)>\mu'$, so $F \in \Coh^{\mu'}_{s,q}(X)$. Otherwise, consider $F \in \Coh^s(X)$ semistable with $\mu_{s,q}(F) \leq \mu$, so $F[1] \in \Coh^{\mu}_{s,q}(X)$. If $\mu_{s,q}(F) \leq \mu'$, then $F[1] \in \Coh^{\mu'}_{s,q}(X)$, while if $\mu_{s,q}(F) > \mu'$, then $F[1] \in \Coh^{\mu'}_{s,q}(X)[1]$. By the definition of $\Coh^{\mu}_{s,q}(X)$, we deduce the claim. 

As a consequence, we have the same relation between the restrictions of the hearts on $\Ku(X)_i$ by \cite[Lemma 4.3]{BLMS}, i.e.\
$$\AA(s,q,\mu) \subset \langle \AA(s,q,\mu'),  \AA(s,q,\mu')[1]\rangle.$$

By definition $Z^{\mu}_{s,q}= \frac{1}{u}Z_{s,q}$ and  $Z^{\mu'}_{s,q}= \frac{1}{u'}Z_{s,q}$, for unit vectors $u, u'$ in the upper half plane. Recall the generators $b_1$ and $b_2$ of $\NN(\Ku(X)_1)$ defined in \eqref{eq_basis}. Since $\Ku(X)_3=\Ku(X)_1(H)$ and $\Ku(X)_2=\L_{\OO_X}(\Ku(X)_3)$, we have that 
\begin{align} \label{eq_d12}
d_1:=b_1(H)=(1, H, \frac{1}{5}H^2, -\frac{5}{6}),\\
d_2:=b_2(H)=(0, H, \frac{2}{5}H^2, -\frac{5}{6}) \nonumber
\end{align}
form a basis of $\NN(\Ku(X)_3)$, and 
\begin{align} \label{eq_c12}
c_1:=(\L_{\OO_X})_*(d_1)=(-3, H, \frac{1}{5}H^2, -\frac{5}{6}),\\
c_2:=(\L_{\OO_X})_*(d_2)=(-4, H, \frac{2}{5}H^2, -\frac{5}{6}) \nonumber
\end{align}
for $\NN(\Ku(X)_2)$. An easy computation shows that multiplying by $1/u$ and $1/u'$ does not change the orientation of the basis $Z_{s,q}(b_1), Z_{s,q}(b_2)$ of $\C$. Thus the basis $Z^{\mu}_{s,q}(b_1)$, $Z^{\mu}_{s,q}(b_2)$ and $Z^{\mu'}_{s,q}(b_1)$, $Z^{\mu'}_{s,q}(b_2)$ have the same orientation. Analogous comments hold for $d_1, d_2$ and $c_1, c_2$. 

Note that $Z^{\mu}_{s,q}= \frac{u'}{u}Z^{\mu'}_{s,q}$, thus setting $M:=\frac{u}{u'}$, we have $Z(s,q,\mu)= M^{-1}Z(s,q,\mu')$ and there exists a cover $\widetilde{g}=(g, M) \in \widetilde{\text{GL}}^+(2,\R)$ such that $\sigma(s,q,\mu') \cdot \widetilde{g}=(\AA', M^{-1}Z(s,q,\mu')=Z(s,q,\mu))$, where $$\AA' \subset \langle \AA(s,q,\mu'), \AA(s,q,\mu')[1]  \rangle.$$
It follows that the stability conditions $\sigma(s,q,\mu)$ and  $\sigma(s,q,\mu') \cdot \widetilde{g}$ have the same central charge and their hearts are tilt of the same heart $\AA(s,q,\mu')$. \cite[Lemma 8.11]{BMS} implies that they are the same stability condition.  
\end{proof}

We end this section by showing that the induced stability conditions on each $\Ku(X)_i$ are in the same orbit with respect to the action of $\widetilde{\text{GL}}^+(2,\R)$. 

\begin{prop} \label{prop_sameorbit}
Fix $i=1,2,3$. The stability conditions induced in item (i) of Proposition \ref{prop_stabcondKu} on $\Ku(X)_i$ are in the same orbit with respect to the $\widetilde{\emph{GL}}^+(2,\R)$- action.
\end{prop}
\begin{proof}
We explain the proof for $i=1$, the other cases are analogous.

Let $(s, q)$, $(s', q')$ as in Proposition \ref{prop_stabcondKu}(1). It is not restrictive to assume $s' \geq s$ and $q \geq q'$. By Lemma \ref{lemma_heartchangingsq} below, we only need to show that the central charges of $\sigma(s, q)$ and $\sigma(s', q')$ are in the same orbits with respect to the action of $\text{GL}^+(2, \R)$. Note that for every $(s, q)$ as in Proposition \ref{prop_stabcondKu}(1), we can choose to tilt at $\mu=-\frac{9}{10}$. Indeed, since $(s,q)$ is below the line $q=-\frac{9}{10}s-\frac{2}{5}$ passing through $\widetilde{v}_H(\OO_X(-H))$ and $\widetilde{v}_H(\UU_X)$, it satisfies the inequalities
\begin{equation*}
\begin{cases}
\mu_{s,q}(\OO_X(-H)[1])=\frac{\frac{1}{2}-q}{-1-s} \leq -\frac{9}{10} \\
\mu_{s,q}(\UU_X)=\frac{\frac{1}{10} -2q}{-1-2s} > -\frac{9}{10}.
\end{cases}    
\end{equation*}
By Lemma \ref{lem_notdependonmu}, the stability condition $\sigma(s,q)$ does not depend on the choice of $\mu$, so we can assume $\mu=-\frac{9}{10}$. In particular, $u=\frac{1}{\sqrt{181}}(9+10\sqrt{-1})$. 

Now consider the central charges $Z_{s,q}^\mu$ and $Z_{s',q'}^{\mu}$. Since multiplying by $\frac{1}{u}$ does not change the orientation, we reduce to compare the orientations of $Z_{s,q}$ and $Z_{s',q'}$ on the basis $b_1$, $b_2$. We have
$$Z_{s,q}(b_1)=10(q+\frac{3}{10})+10\sqrt{-1}(-s), \quad Z_{s,q}(b_2)= 10(\frac{3}{5}+\sqrt{-1}).$$ 
Then 
$$
\begin{vmatrix}
q+\frac{3}{10} & \frac{3}{5}\\
-s & 1
\end{vmatrix}= q+\frac{3}{5}s+\frac{3}{10} > \frac{1}{2}s^2 + \frac{3}{5}s + \frac{9}{40}>0,
$$
since $(s, q)$ is above the parabola $q=\frac{1}{2}s^2-\frac{3}{40}$. In particular, there exists $N \in \text{GL}^+(2, \R)$ such that $Z_{s', q'}= N^{-1} \cdot Z_{s,q}$. We write $N=\frac{1}{\text{det}(N^{-1})}\begin{pmatrix}
a & b \\
c & d
\end{pmatrix}$, where
$$a=\frac{10q +3 + 6s'}{10q+3+6s}, \quad b=\frac{6(q'-q)}{10q+3+6s}$$
$$c=\frac{10(s'-s)}{10q+3+6s}, \quad d=\frac{10q'+3+6s}{10q+3+6s}.$$
As a consequence, we have $Z_{s',q'}^{\mu}= M^{-1} Z_{s,q}^{\mu}$, where $M^{-1}= \frac{1}{u}N^{-1}u$ and there exists $(g, M) \in \widetilde{\text{GL}}^+(2, \R)$ such that $g(0, 1) \subset (0,2)$. This implies $\sigma(s,q) \cdot (g,M)= (\A', Z^{\mu}_{s',q'}|_{\NN(\Ku(X))_1})$ with
$$\A' \subset \langle \A(s,q), \A(s,q)[1] \rangle.$$
Thus the stability conditions $\sigma(s',q')$ and $\sigma(s, q) \cdot (g,M)$ have the same central charge and their hearts are tilt of $\A(s,q)$. We conclude that they are the same stability condition by \cite[Lemma 8.11]{BMS}.
\end{proof}

\begin{lemma} \label{lemma_heartchangingsq}
Fix $i=1,2,3$. Let $(s,q)$, $(s', q')$ as in Proposition \ref{prop_stabcondKu}(i). If $s<s'$, then $\A(s', q') \subset \langle \A(s,q), \A(s,q)[1] \rangle$, while if $s=s'$, then $\A(s', q')= \A(s,q)$.
\end{lemma}
\begin{proof}
The argument is similar to the one used in the proof of \cite[Lemma 3.8]{PY}. Consider the case $i=1$, the other are analogous. By Lemma \ref{lem_notdependonmu} we can fix $\mu=-\frac{9}{10}$. We denote by $\PP_{s,q}$ the slicing defined by $\sigma_{s,q}$. We claim that $\Coh^{\mu}_{s,q}(X)= \PP_{s,q}(\phi_u, \phi_u+1]$, where $\phi_u=\frac{1}{\pi} \text{arg}(u)$. Indeed, assume $E \in \Coh^{\mu}_{s,q}(X)$ is $\sigma^{\mu}_{s,q}$-semistable. Then there is a triangle 
$A[1] \to E \to B$,
where $A \in \Coh^s(X)$ (resp.\ $B \in \Coh^s(X)$) and its $\sigma_{s,q}$-semistable factors have slope $\mu_{s,q} \leq \mu$ (resp.\ $>\mu$). If $Z_{s,q}(B) \neq 0$, then $A[1]$ has larger slope than $B$ with respect to $\sigma^{\mu}_{s,q}$. This would contradict the semistability of $E$, unless either $E=B$, or $E=A[1]$. If $E=B$, the $\sigma_{s,q}$-semistable factors of $E$ would have phase in the interval $(\phi_u, 1]$ by definition of $B$. Actually, this also shows $E$ is $\sigma_{s,q}$-semistable, as a destabilizing sequence of $E$ with respect to $\sigma_{s,q}$ would destabilize $E$ with respect to $\sigma^{\mu}_{s,q}$. A similar observation shows that $A[1] \in \PP_{s,q}(1, \phi_u+1]$. It remains to consider the case when $Z_{s,q}(B) = 0$, i.e.\ $B$ is a torsion sheaf supported on points. Then $B$ is $\sigma_{s,q}$-semistable of phase $1$. Since $A[1] \in \PP_{s,q}(1, \phi_u+1]$, we conclude that $E \in \PP_{s,q}(\phi_u, \phi_u+1]$. This shows $\Coh^\mu_{s,q}(X) \subset \PP_{s,q}(\phi_u, \phi_u+1]$.   Since those are hearts of bounded t-structures, we deduce that they are equal.

Now if $s'>s$, it is easy to see that $\Coh^{s'}(X)$ is a tilt of $\Coh^s(X)$, i.e.\ $\Coh^{s'}(X) \subset \langle \Coh^s(X), \Coh^s(X)[1] \rangle$. Equivalently, $\PP_{s',q'}(0, 1] \subset \PP_{s,q}(0,2]$. The action by multiplication with $u^{-1}$ preserves the distance of the slicings $\PP_{s,q}$ and $\PP_{s',q'}$, thus 
$$\Coh^{\mu}_{s',q'}(X)= \PP_{s',q'}(\phi_u, \phi_u+1] \subset \PP_{s,q}(\phi_u, \phi_u+2]= \langle \Coh^{\mu}_{s,q}(X), \Coh^{\mu}_{s,q}(X)[1] \rangle.$$
Consider $\A(s',q')= \Ku(X) \cap \Coh^{\mu}_{s',q'}(X)$. Since the cohomology with respect to the restricted heart of an objects $E \in \Ku(X)$ is the same as the cohomology in $\Coh^{\mu}_{s,q}(X)$ by \cite[Lemma 4.3]{BLMS}, we deduce that
$$\A(s', q') \subset \langle \A(s,q), \A(s, q)[1] \rangle.$$
If $s'=s$, we get $\Coh^{\mu}_{s',q'}(X)=\Coh^{\mu}_{s,q}(X)$, which implies $\A(s',q')=\A(s,q)$.
\end{proof}

\textbf{Notation:} In the next, we will use the subscript $s,q$ (resp.\ $\alpha, \beta$) when we refer to the reparametrized tilt stability condition (resp.\ to the classical tilt stability). If we work in the region above the parabola $q-\frac{1}{2}s^2=0$, we will prefer to use the classical tilt stability condition depending on $\alpha$ and $\beta$, and we will make use of the tilt stability below this parabola and above Li's boundary only where it is necessary.

We will denote by $\Coh^s(X)_{\mu_{s,q}>\mu}$ (resp.\ $\Coh^s(X)_{\mu_{s,q}\leq \mu}$) the subcategory of $\Coh^s(X)$ generated by $\mu_{s,q}$-semistable objects with slope $\mu_{s,q} > \mu$ (resp.\ $\leq \mu)$, and analogous notation with the subscript $\alpha, \beta$.

\subsection{Proof of Theorem \ref{thm_main}} \label{sec_proof}

Consider $\Ku(X)_3$ defined in \eqref{eq_Ku3}. By \cite[lemma 2.6]{Kuz_calabi}, since $S_X(-)=-\otimes \OO_X(-H)[3]$, the Serre functor $S_{\Ku(X)_3}$ on $\Ku(X)_3$ satisfies
\begin{equation}
\label{eq_defofSerrefunctor}    
S_{\Ku(X)_3}^{-1}(-)= \L_{\UU_X^\vee} \circ \L_{\OO_X(H)} \circ (- \otimes \OO_X(H))[-3]= (- \otimes \OO_X(H)) \circ \L_{\UU_X} \circ \L_{\OO_X}[-3].
\end{equation}
The goal of the next sections is to prove Theorem \ref{thm_main}, which follows from the result below.

\begin{thm}
\label{thm_mainsec3}
Let $\sigma(s_3,q_3)$ be a stability condition on $\Ku(X)_3$ as induced in Proposition \ref{prop_stabcondKu}(3). Then there exists $\widetilde{g} \in \widetilde{\emph{GL}}^+_2(\R)$ such that $$S^{-1}_{\Ku(X)_3} \cdot \sigma(s_3,q_3)= \sigma(s_3,q_3) \cdot \widetilde{g}.$$
\end{thm}
\noindent Then in Corollary \ref{cor_stabonfourfold} we show more precisely that $S_{\Ku(X)_3}[-2] \cdot \sigma(s_3,q_3)= \sigma(s_3,q_3)$, completing the proof of Theorem \ref{thm_main}.

Here we outline the strategy of the proof of Theorem \ref{thm_mainsec3}. The idea is to decompose $S^{-1}_{\Ku(X)_3}$ as in \eqref{eq_defofSerrefunctor} and study the action of $\L_{\OO_X}$ on $\sigma(s_3,q_3)$ and then of $\L_{\UU_X}$ on $\L_{\OO_X} \cdot \sigma(s_3,q_3)$. In fact, $\L_{\OO_X}$ (resp.\ $\L_{\UU_X}$) induces an equivalence between $\Ku(X)_3$ and $\Ku(X)_2$ (resp.\ $\Ku(X)_2$ and $\Ku(X)_1$), so $\L_{\OO_X} \cdot \sigma(s_3,q_3)$ and $\L_{\UU_X} \cdot \L_{\OO_X} \cdot \sigma(s_3,q_3)$ are stability conditions on $\Ku(X)_2$ and $\Ku(X)_1$, respectively. 

First, in Section \ref{sec_stabonKu3} we consider special values of $s_3$ and $q_3$ very close to $0$. Here it is not necessary to work with the reparametrized tilt stability conditions, so we use the notation with $\alpha$ and $\beta$. In particular we consider the stability condition $\sigma(\alpha,\epsilon)$, for $\epsilon>0$ very small and $0 < \alpha < \epsilon$. In Lemma \ref{lemma_leftmutationheart} we show that the heart $\L_{\OO_X}(\A(\alpha,\epsilon))$ on $\Ku(X)_2$ is a tilting of $\A(\alpha',-\epsilon')$ for $0< \epsilon' \leq \epsilon$, $0< \alpha' < \epsilon'$. The basic idea is that when moving from $\epsilon$ to $-\epsilon'$, the only problematic object in $\Coh^0_{\alpha,\epsilon}(X)$ is $\OO_X[2]$, which belongs to $\Coh^0_{\alpha',-\epsilon'}(X)[2]$. Then we show in Proposition \ref{prop_leftmutationstab} that the stability condition $\L_{\OO_X} \cdot \sigma(\alpha,\epsilon)$ on $\Ku(X)_2$ is the same as $\sigma(\alpha',-\epsilon')$ up to the $\widetilde{\text{GL}}^+_2(\R)$-action. This implies the same statement for every stability condition $\sigma(s_3,q_3)$ on $\Ku(X)_3$ (see Corollary \ref{cor_leftmutationstab}).

Next, in Section \ref{sec_stabonKu2} we follow the same argument for the stability conditions $\sigma(s_2,q_2)$ on $\Ku(X)_2$ and the left mutation $\L_{\UU_X}$. Here we need to work with the stability conditions over Li's boundary, as we need to consider $s_2$ very close to $-\frac{1}{2}$ and $q_2$ close to $\frac{1}{20}$. Analogously, we show in Lemma \ref{lemma_leftmutationheartafterU} that the heart $\L_{\UU_X}(\A(-\frac{1}{2}+\epsilon,q_2))$ on $\Ku(X)_1$ is a tilt of $\A(-\frac{1}{2}-\epsilon',q_2')$. This allows to show in Corollary \ref{cor_leftmutationstabafterU} that $\L_{\UU_X} \cdot \sigma(s_2,q_2)$ on $\Ku(X)_1$ is in the same orbit with respect to the $\widetilde{\text{GL}}^+_2(\R)$-action of the induced stability conditions $\sigma(s_1,q_1)$ on $\Ku(X)_1$.

Finally, we simply observe that acting via $(-) \otimes \OO_X(H)$ on a stability condition $\sigma(s_1,q_1)$ on $\Ku(X)_1$, we get $\sigma(s_1+1,q_1')$, namely a stability condition on $\Ku(X)_3$ in the same orbit of $\sigma(s_3,q_3)$. 

\subsection{Stability conditions on $\Ku(X)_3$ and action of $\L_{\OO_X}$} \label{sec_stabonKu3}

In this section we study the action of $\L_{\OO_X}$ on the stability conditions $\sigma(s_3,q_3)$ on $\Ku(X)_3$ defined in Proposition \ref{prop_stabcondKu}(3). The main result is Corollary \ref{cor_leftmutationstab}.

We start by considering $(s_3, q_3)$ close to $(0,0)$. For this reason, we can simply work with the usual parametrization of the tilt stability $\sigma_{\alpha,\beta}$ and $\beta=\epsilon >0$ very small. 

\begin{lemma}
\label{lemma_coh_epsilon}
Fix $\epsilon>0$ very small. Assume that $F \in \emph{Coh}^0_{\alpha, \epsilon}(X)$ for every $0 < \alpha < \epsilon$. Then there exist $\epsilon'>0$ very small and $0 < \alpha' < \epsilon'$ such that 
$$F \in \langle \emph{Coh}^0_{\alpha', -\epsilon'}(X), \emph{Coh}^0_{\alpha', -\epsilon'}(X)[1], \OO_X[2] \rangle.$$
%There exist $\epsilon >0$, $0 < \alpha < \epsilon$ such that 
%$$\emph{Coh}^0_{\alpha, \epsilon}(X) \subset \langle \emph{Coh}^0_{\alpha, -\epsilon}(X), \emph{Coh}^0_{\alpha, -\epsilon}(X)[1], \OO_X[2] \rangle.$$
\end{lemma}
\begin{proof}
\textbf{Step 1:} We show that if $E \in \Coh^\epsilon(X)$ such that $\mu^+_{\alpha, \epsilon}(E) \leq 0$, then there exists $0 <\epsilon' \leq \epsilon$ such that $E$ is an extension of objects in $\Coh^{-\epsilon'}(X)$ and objects of the form $\GG[1]$, where $\GG$ is a slope semistable coherent sheaf with $\mu_H(\GG)=0$ and $\mu_{\alpha, \epsilon}^+(\GG[1]) = \mu_{\alpha,\epsilon}(\OO_X[1])$.

Consider $E$ as above. Note that $\mu_{\alpha, \epsilon}(\OO_X[1])=\frac{-\epsilon^2+\alpha^2}{2\epsilon} < 0$ for $\alpha < \epsilon$ and converges to $0$ for $\alpha \to \epsilon$. Thus, up to taking $\alpha$ close to $\epsilon$, we can assume that $\mu^+_{\alpha, \epsilon}(E) \leq \mu_{\alpha,\epsilon}(\OO_X[1])$. On the other hand, by definition $E$ is an extension of the form
$$\HH^{-1}(E)[1] \to E \to \HH^0(E),$$
where $\HH^0(E)$ (resp.\ $\HH^{-1}(E)$) is in $\Coh(X)$ and its slope semistable factors have slope $\mu_H >\epsilon$ (resp.\ $\leq \epsilon$). Clearly, $\HH^0(E) \in \Coh^{-\epsilon'}(X)$ for every $0<\epsilon'\leq \epsilon$. Consider $\HH^{-1}(E)$ and denote by $\GG_1, \dots, \GG_k$ its slope semistable factors. Then there exists $0 < \epsilon' \leq \epsilon$ such that if $\mu_H(\GG_i) <0$, then $\mu_H(\GG_i) \leq -\epsilon'$. Then $\GG_i[1] \in \Coh^{-\epsilon'}(X)$.

Assume there is an index $i$ such that $0 \leq \mu_H(\GG_i) \leq \epsilon$. Set $\ch(\GG_i)_{\leq 2}=(d, eH, \frac{f}{2}H^2)$ for integers $d, e, f$. Then $d>0$, $e\geq 0.$ Since  $\mu^+_{\alpha, \epsilon}(E) \leq \mu_{\alpha,\epsilon}(\OO_X[1])$, it follows that
\begin{equation}
\label{eq_slopeGi}  
\mu_{\alpha, \epsilon}(\GG_i[1]) \leq \mu_{\alpha, \epsilon}^+(\GG_i[1]) \leq \mu^+_{\alpha, \epsilon}(E) \leq \mu_{\alpha,\epsilon}(\OO_X[1]).
\end{equation}
However, the point $(s(\GG_i), q(\GG_i))$ corresponding to $\GG_i$ in the affine plane $\mathbb{A}^2_{\R}$ does not belong to $R_{\frac{3}{20}}$ by \cite[Theorem 0.3]{Li_Fano3} (see \eqref{eq_regionR}; equivalently, $f \leq 0$. This implies that
\begin{equation}
\label{eq_somecomputation}
-f\epsilon +\epsilon^2e + \alpha^2e \geq 0, 
\end{equation}
with equality if and only if $e=f=0$. The inequality \eqref{eq_somecomputation} contradicts $\mu_{\alpha, \epsilon}(\GG_i[1]) \leq \mu_{\alpha, \epsilon}(\OO_X[1])$ unless $e=f=0$ when it is an equality. This implies the claim in Step 1.

\textbf{Step 2:} We improve the computation in Step 1, by showing the objects of the form $\GG[1]$ are extensions of copies of $\OO_X[1]$.% and shifts by $1$ of $\sigma_{\alpha',-\epsilon'}$-semistable objects with slope $\mu_{\alpha',-\epsilon'} \leq 0$. 

Note that $\GG[1]$ is a slope semistable torsion free sheaf. Moreover, $\GG[1]$ is $\sigma_{\alpha, \epsilon}$-semistable with slope $\mu_{\alpha,\epsilon}(\OO_X[1])$, since the inequalities in \eqref{eq_slopeGi} are equalities. Since by \cite[Theorem 0.1]{Li_Fano3}, Conjecture 4.1 of \cite{BMS} holds, it follows that $\ch_3(\GG):=g \geq 0$. 

It is not restrictive to assume $\GG$ is slope stable, up to replacing it with one of its stable factors. We claim that $\GG \cong \OO_X$. Indeed, note that $\chi(\OO_X, \GG)=d+g>0$. Thus $\hom(\OO_X, \GG)+ \hom(\OO_X, \GG[2])>0$. We observe that $\GG$ is a reflexive sheaf. Indeed, consider the short exact sequence
$$0 \to \GG \to \GG^{\vee\vee} \to T \to 0$$
of coherent sheaves, where $T$ is torsion supported in dimension $\leq 1$. If $T \neq 0$, then $T$ would destabilize $\GG[1]$ with respect to $\sigma_{\alpha, \epsilon}$ giving a contradiction. In particular, $\GG[1]$ is $\sigma_{\alpha, 0}$-stable (see \cite[Proposition 4.18]{Soheyla}). Thus there exists $\delta>0$ very small such that $\GG \in \Coh^{-\delta}(X)$ is $\sigma_{\alpha, -\delta}$-semistable for some $0 < \alpha < \delta$. By Serre duality, we have
$$\Hom(\OO_X, \GG[2])=\Hom(\GG, \OO_X(-H)[1])=0,$$
where the last equality follows from the fact that $\GG$ and $\OO_X(-H)[1]$ are $\sigma_{\alpha, -\delta}$-semistable in $\Coh^{-\delta}(X)$ with $\mu_{\alpha, -\delta}(\OO_X(-H)[1])< 0 < \mu_{\alpha, -\delta}(\OO_X)=\mu_{\alpha, -\delta}(\GG)$. Thus there exists a non-zero morphism $\OO_X \to \GG$. Since both are slope stable, we conclude that $\OO_X \cong \GG$.

\textbf{Step 3:} We can now prove the statement of the lemma. 

Consider $F \in \Coh^0_{\alpha, \epsilon}(X)$ for $\epsilon >0$ very small and for every $0 < \alpha < \epsilon$. By definition $F$ is an extension of the form
$$A[1] \to F \to B$$
where $B$ (resp.\ $A$) belongs to $\Coh^\epsilon(X)_{\mu_{\alpha,\epsilon}>0}$ (resp.\ $\Coh^\epsilon(X)_{\mu_{\alpha,\epsilon}\leq 0}$). In the next, we show that
\begin{equation}
\label{eq_B}
B \in \langle \Coh^{-\epsilon'}(X)_{\mu_{\alpha',-\epsilon'}>0}, \Coh^{-\epsilon'}(X)[1] \rangle    
\end{equation}
and
\begin{equation}
\label{eq_A}    
A[1] \in \langle \Coh^{-\epsilon'}(X)[1], \OO_X[2] \rangle,
\end{equation}
for $\epsilon' >0$ very small and $0 < \alpha'< \epsilon'$, which imply the statement.

By Step 2, we have that $A[1]$ is an extension of objects in $\Coh^{-\epsilon'}(X)[1]$ and copies of $\OO_X[2]$. This implies \eqref{eq_A}.

Up to replacing $B$ with its semistable factors, we can assume that $B$ is $\sigma_{\alpha, \epsilon}$-semistable. Consider $0< \epsilon' \leq \epsilon$, $0< \alpha' < \epsilon'$ and set $q':=\frac{(\alpha')^2+(\epsilon')^2}{2}$. By a similar argument as that at the beginning of Step 1, we have that $B \in \langle \Coh^{-\epsilon'}(X), \Coh^{-\epsilon'}(X)[1] \rangle$. We will now use the result in \cite{LiZhao2} to control the slope of the tilt semistable factors of $B$ when we deform the tilt stability condition $\sigma_{\alpha, \epsilon}$ to $\sigma_{\alpha', -\epsilon'}$. To this end, write $\ch(B)_{\leq 2}=(a, bH, \frac{c}{2}H^2)$ for some integers $a, b, c$. We denote by $\ell$ the line connecting the point $(\epsilon, \delta:=\frac{\alpha^2+\epsilon^2}{2})$ to the point $(\frac{b}{a}, \frac{c}{2a})$ corresponding to $B$ in the affine plane. Let $B^+$ and $B^-$ be the intersection points of the parabola $q= \frac{1}{2}s^2$ with $\ell$. By \cite[Lemma 3]{LiZhao2}, if $B_i$ is a $\sigma_{\alpha', -\epsilon'}$-semistable factor of $B$, then its slope satisfies 
$$\mu_{\alpha', -\epsilon'}(B^-) \leq \mu_{\alpha', -\epsilon'}(B_i) \leq \mu_{\alpha', -\epsilon'}(B^+).$$

We claim that the coordinate $s(B^-)$ of $B^-$ on the $s$-axis is $>0$. Indeed, $\ell$ has equation 
$$q= \frac{\frac{c}{2}-a\delta}{b-a\epsilon}(s-\epsilon) +\delta,$$
so the $s$-coordinate of $B^{\pm}$ is given by the solutions of 
\begin{equation} \label{eq_quadratic}
\frac{1}{2}s^2 - \frac{\frac{c}{2}-a\delta}{b-a\epsilon}s+ \frac{\frac{c}{2}-a\delta}{b-a\epsilon}\epsilon -\delta=0.    
\end{equation}
Recall that $\mu_{\alpha, \epsilon}(B)>0$, equivalently $\mu_{\epsilon, \delta}(B)= \frac{\frac{c}{2}-a\delta}{b-a\epsilon} > \epsilon$ (see \eqref{eq_reparametrizedslope} for the relation between reparametrized tilt stability). Then the coefficient of $s$ in \eqref{eq_quadratic} is negative, while $\frac{\frac{c}{2}-a\delta}{b-a\epsilon}\epsilon -\delta> \epsilon^2-\delta>0$, as $\delta < \epsilon^2$ (equivalently $\alpha< \epsilon)$. It follows that $\eqref{eq_quadratic}$ has two positive solutions. 

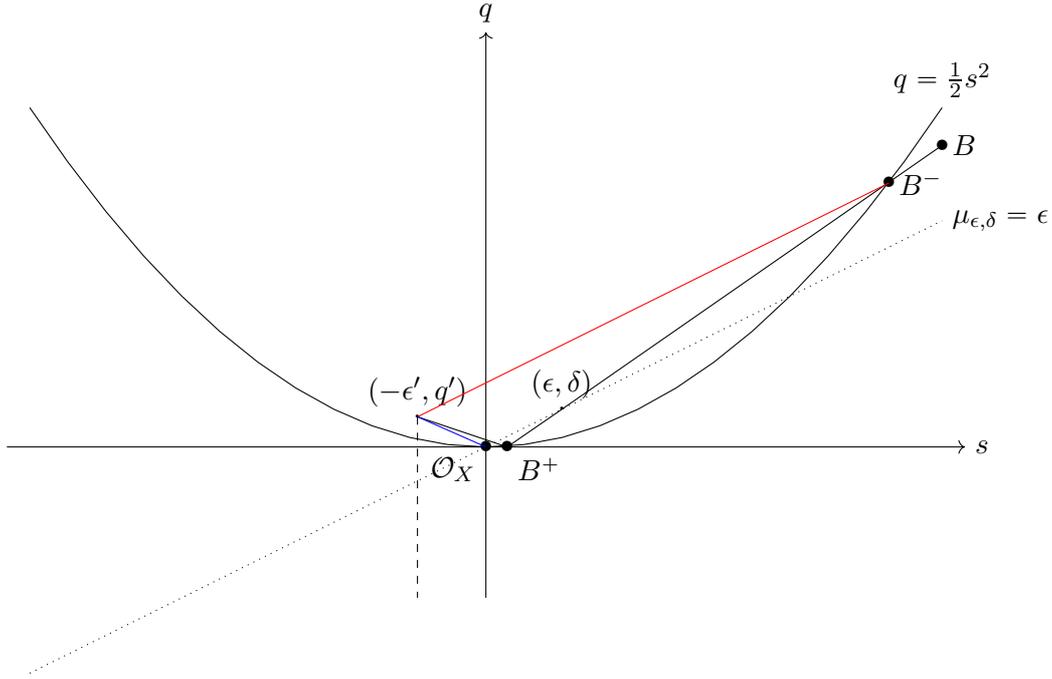
\begin{figure}[htb]
\centering
\begin{tikzpicture}[domain=-6:6]
\draw[->] (-6.3,0) -- (6.3,0) node[right] {$s$};
\draw[->] (0,-2) -- (0, 5.5) node[above] {$q$};
\draw plot (\x,{0.125*\x*\x}) node[above] {$q = \frac{1}{2}s^{2}$};
\coordinate (O) at (0,0);
\node at (O) {$\bullet$};
\draw (O) node [below left]  {$\OO_X$};

\draw[dotted] plot (\x,{0.5*\x}) node[right] {$\mu_{\epsilon, \delta} = \epsilon$};
 
\coordinate (B) at (6,4); 
\node at (B) {$\bullet$};
\draw (B) node [right]  {$B$};

\coordinate (P) at (1,0.5); 
\node at (P) {$\cdot$};
\draw (P) node [above]  {$(\epsilon,\delta)$};

\coordinate (Q) at (-0.9,0.4); 
\node at (Q) {$\cdot$};
\draw (Q) node [above]  {$(-\epsilon',q')$};
\draw[dashed] (Q) -- (-0.9,-2);

\draw (B) -- (P);
\draw (P) -- (0.28,0);

\coordinate (B+) at (0.28,0); 
\node at (B+) {$\bullet$};
\draw (B+) node [below right]  {$B^+$};

\coordinate (B-) at (5.3,3.5); 
\node at (B-) {$\bullet$};
\draw (B-) node [right]  {$B^-$};

\draw[red] (Q) -- (B-);
\draw (Q) -- (B+);
\draw[blue] (Q) -- (0,0);

\end{tikzpicture}

\caption{The tilt-stability $\sigma_{\alpha,\epsilon}$ (resp.\ $\sigma_{\alpha',-\epsilon'}$) corresponds to the point $(\epsilon, \delta:=\frac{\alpha^2+\epsilon^2}{2})$ (resp.\ $(-\epsilon', q':=\frac{(\alpha')^2+(\epsilon')^2}{2})$) in the affine plane. Tilting at $\mu_{\alpha,\epsilon}=0$ (resp.\ $\mu_{\alpha',\epsilon'}=0$) is equivalent to tilting at $\mu_{\epsilon,\delta}= \epsilon$ (resp.\ $\mu_{-\epsilon',q'}= -\epsilon'$). Let $B$ be a $\sigma_{\epsilon, \delta}$-semistable object in $\Coh^\epsilon(X)$ with $\mu_{\epsilon, \delta}(B) > \epsilon$. By \cite[Lemma 3]{LiZhao2} the slope of its $\sigma_{-\epsilon', q'}$-semistable factors is greater or equal than the slope of $B^-$, which is the gradient of the red line connecting $(-\epsilon', q')$ to $B^-$. Since $B^-$ has positive $s$-coordinate, its slope with respect to $\sigma_{-\epsilon', q'}$ is larger than $\mu_{-\epsilon', q'}(\OO_X)$, which is represented in blue.  \label{fig_3}}  
\end{figure}

Now, comparing the slopes using Figure \ref{fig_3}, we deduce that
$$\mu_{-\epsilon', q'}(B^-) > \mu_{-\epsilon', q'}(\OO_X) > -\epsilon',$$
equivalently
$$\mu_{\alpha', -\epsilon'}(B^-) > \mu_{\alpha', -\epsilon'}(\OO_X) > 0.$$
This proves \eqref{eq_B} and ends the proof of the lemma.
\end{proof}

\begin{lemma}
\label{lemma_leftmutationO}
Let $\epsilon>0$ be very small. There exists $0 < \alpha < \epsilon$ such that if $F \in \CCoh^0_{\alpha,\epsilon}(X)$, then $\L_{\OO_X}(F)$ is in $\CCoh^0_{\alpha,\epsilon}(X)$.
\end{lemma}
\begin{proof}
As in \cite[Lemma 5.9]{BLMS} we have the five terms exact sequence 
$$0 \to \HH^{-1}(\L_{\OO_X}(F)) \to \OO_X[2]^{\oplus k_0} \to F \to \HH^0(\L_{\OO_X}(F)) \to \OO_X[2]^{\oplus k_1} \to 0,$$
where $\HH^i(\L_{\OO_X}(F))$ denotes the cohomology in $\Coh^0_{\alpha,\epsilon}(X)$ and $k_0, k_1$ are integers. Note that
$$\mu^0_{\alpha,\epsilon}(\OO_X[2])=\frac{2 \epsilon}{\epsilon^2-\alpha^2} \to +\infty \text{ if } \alpha \to \epsilon.$$
Thus up to replacing $\alpha$, we can assume that $\OO_X[2]$ is the stable factor of $F$ with larger slope. It follows that $\HH^{-1}(\L_{\OO_X}(F))=0$, thus $\L_{\OO_X}(F) \in \Coh^0_{\alpha,\epsilon}(X)$. 
\end{proof}

\begin{lemma}
\label{lemma_leftmutationheart}
Let $\epsilon>0$ be very small. Then there exist $0 < \alpha < \epsilon$, $0 < \epsilon' \leq \epsilon$, $0< \alpha'< \epsilon'$ such that
$$\L_{\OO_X}(\A(\alpha,\epsilon)) \subset \langle \AA(\alpha',-\epsilon'), \A(\alpha',-\epsilon')[1] \rangle.$$
\end{lemma}
\begin{proof}
Consider $F \in \A(\alpha,\epsilon)$ and its left mutation $\L_{\OO_X}(F)$. By Lemma \ref{lemma_leftmutationO} we can find $\alpha$ such that $\L_{\OO_X}(F) \in \Coh^0_{\alpha,\epsilon}(X)$. Note that $F \in \AA(\alpha'', \epsilon)$ for every $0 < \alpha''< \epsilon$ by Lemma \ref{lemma_heartchangingsq}. By Lemma \ref{lemma_coh_epsilon}, there exist $\epsilon'>0$ very small and $0 < \alpha' < \epsilon'$ such that
$$\L_{\OO_X}(F) \in \langle \Coh^0_{\alpha',-\epsilon'}(X), \Coh^0_{\alpha',-\epsilon'}(X)[1], \OO_X[2] \rangle.$$
Up to taking a smaller $\epsilon'$, the above relation holds for every $F \in \AA(\alpha,\epsilon)$. Note that $\A(\alpha',-\epsilon')=\Coh^0_{\alpha',-\epsilon'}(X) \cap \Ku(X)_2$ is a heart of a stability condition on $\Ku(X)_2$ by Proposition \ref{prop_stabcondKu}(2). Since $\L_{\OO_X}(F)$ is in $\Ku(X)_2$ and by \cite[Lemma 4.3]{BLMS}, its cohomology in $\Coh^0_{\alpha',-\epsilon'}(X)$ belongs to $\Ku(X)_2$ as well, we deduce that $\L_{\OO_X}(F) \in  \langle \AA(\alpha',-\epsilon'), \A(\alpha',-\epsilon')[1] \rangle$. Finally, we observe that the statement does not depend on $\alpha' < \epsilon'$, by Lemma \ref{lemma_heartchangingsq}.
\end{proof}

\begin{prop}
\label{prop_leftmutationstab}
Let $\epsilon>0$ be very small. Then there exist $0 < \alpha < \epsilon$, $0 < \epsilon' \leq \epsilon$, $0< \alpha'< \epsilon'$ such that there exists $\widetilde{g} \in \widetilde{\emph{GL}}^+_2(\R)$ satisfying
$$\L_{\OO_X} \cdot \sigma(\alpha,\epsilon)= \sigma(\alpha',-\epsilon') \cdot \widetilde{g}.$$
\end{prop}
\begin{proof}
Recall that the stability condition $\L_{\OO_X} \cdot \sigma(\alpha,\epsilon)$ has heart $\L_{\OO_X}(\AA(\alpha,\epsilon))$ and stability function $Z':=Z(\alpha,\epsilon) \circ (\L_{\OO_X})_*^{-1}$. As done for instance in Proposition \ref{prop_sameorbit}, we can check that there exists $\widetilde{g} \in \widetilde{\text{GL}}^+_2(\R)$ such that $\sigma(\alpha',-\epsilon') \cdot \widetilde{g} =\sigma'$, where $\sigma'=(\AA',Z')$ and $\AA'$ is a tilt of $\AA(\alpha',-\epsilon')$, up to shifting. More precisely, one first needs to check there exists $M \in \text{GL}_2^+(\R)$ such that $Z'=M^{-1} \cdot Z(\alpha', -\epsilon')$, or equivalently, 
\begin{equation} \label{eq_M}
Z(\alpha,\epsilon)=M^{-1} \cdot Z(\alpha', -\epsilon') \cdot (\L_{\OO_X})_*.    
\end{equation}
In order to do this, recall the basis $d_1:=b_1(H)$, $d_2:=b_2(H)$ 
of $\NN(\Ku(X)_3)$ given in \eqref{eq_d12} and the basis $c_1:=(\L_{\OO_X})_*(d_1)$, $c_2:=(\L_{\OO_X})_*(d_2)$ of $\NN(\Ku(X)_2)$ defined in \eqref{eq_c12}. Recall also that $Z(\alpha, \epsilon)=-\sqrt{-1}Z_{\alpha, \epsilon}|_{\NN(\Ku(X)_3)}$, $Z(\alpha', -\epsilon')=-\sqrt{-1}Z_{\alpha', -\epsilon'}|_{\NN(\Ku(X)_2)}$ as defined in \eqref{eq_defodZsq}. Then
\begin{align*}
Z(\alpha, \epsilon)(d_1)&=(1-\epsilon) + \sqrt{-1}(\frac{1}{5}-\epsilon +\frac{\epsilon^2}{2}- \frac{\alpha^2}{2}),\\
Z(\alpha, \epsilon)(d_2)&=1 + \sqrt{-1}(\frac{2}{5}-\epsilon),\\
Z(\alpha',-\epsilon')(c_1)&=(1-3\epsilon')+\sqrt{-1}(\frac{1}{5}+\epsilon'-\frac{3}{2}(\epsilon')^2+\frac{3}{2}(\alpha')^2),\\
Z(\alpha',-\epsilon')(c_2)&=(1-4\epsilon') +\sqrt{-1} (\frac{2}{5}+\epsilon' -2(\epsilon')^2+2(\alpha')^2).
\end{align*}
The two matrices having on the columns the components with respect to the standard basis of $\C=\R^2$ of $Z(\alpha, \epsilon)(d_1)$, $Z(\alpha, \epsilon)(d_2)$, and $Z(\alpha', -\epsilon')(c_1)$, $Z(\alpha', -\epsilon')(c_2)$ respectively, have determinant $$\frac{\alpha^2}{2}+\frac{\epsilon^2}{2}-\frac{2}{5}\epsilon+\frac{1}{5}>0, \, \frac{(\alpha')^2}{2}+\frac{(\epsilon')^2}{2}-\frac{2}{5}\epsilon'+\frac{1}{5}>0.$$ Thus there exists $M \in \text{GL}^+_2(\R)$ satisfying \eqref{eq_M}. More explicitly, setting $M^{-1}=
\begin{pmatrix} 
x_1 & x_2\\
x_3 & x_4
\end{pmatrix}$, condition \eqref{eq_M} translates into
$$
\begin{cases}
Z(\alpha,\epsilon)(d_1)=M^{-1} \cdot Z(\alpha', -\epsilon') \cdot (\L_{\OO_X})_*(d_1),\\
Z(\alpha,\epsilon)(d_2)=M^{-1} \cdot Z(\alpha', -\epsilon') \cdot (\L_{\OO_X})_*(d_2)
\end{cases}
$$
which is equivalent to solve the linear system 
$$
\begin{cases}
(1-3\epsilon')x_1 + (\frac{1}{5}+\epsilon'-\frac{3}{2}(\epsilon')^2+\frac{3}{2}(\alpha')^2)x_2= 1-\epsilon \\
(1-3\epsilon')x_3 + (\frac{1}{5}+\epsilon'-\frac{3}{2}(\epsilon')^2+\frac{3}{2}(\alpha')^2)x_4= \frac{1}{5}-\epsilon +\frac{\epsilon^2}{2}-\frac{\alpha^2}{2}\\
(1-4\epsilon')x_1 + (\frac{2}{5}+\epsilon' -2(\epsilon')^2+2(\alpha')^2)x_2=1 \\
(1-4\epsilon')x_3 + (\frac{2}{5}+\epsilon' -2(\epsilon')^2+2(\alpha')^2)x_4= \frac{2}{5}-\epsilon. 
\end{cases}  
$$ 
Using a computer, we can find the solution of the above linear system and check the existence of a cover $\widetilde{g}=(g, M) \in \widetilde{\text{GL}}^+_2(\R)$ with the desired properties.
 
Since by Lemma \ref{lemma_leftmutationheart}, the heart $\L_{\OO_X}(\AA(\alpha,\epsilon))$ is a tilt of $\AA(\alpha',-\epsilon')$, by \cite[Lemma 8.11]{BMS} we conclude $\sigma'=\L_{\OO_X} \cdot \sigma(\alpha,\epsilon)$, as we wanted.
\end{proof}

\begin{cor}
\label{cor_leftmutationstab}
For $i=2,3$, if $\sigma(s_i,q_i)$ is a stability condition on $\Ku(X)_i$, then there exists $\widetilde{g} \in \widetilde{\emph{GL}}^+_2(\R)$ such that
$$\L_{\OO_X} \cdot \sigma(s_3,q_3)= \sigma(s_2,q_2) \cdot \widetilde{g}.$$
\end{cor} 
\begin{proof}
By Proposition \ref{prop_sameorbit}, we have that $\sigma(s_3,q_3)$ (resp.\ $\sigma(s_2,q_2)$) is in the same orbit of $\sigma(\alpha, \epsilon)$ (resp.\ $\sigma(\alpha', -\epsilon')$) with respect to the $\widetilde{\text{GL}}^+_2(\R)$-action. Since the action of $\L_{\OO_X}$ commutes with the $\widetilde{\text{GL}}^+_2(\R)$-action, by Proposition \ref{prop_leftmutationstab}, we deduce the claim.
\end{proof}

\subsection{Stability conditions on $\Ku(X)_2$ and action of $\L_{\UU_X}$} \label{sec_stabonKu2}

Our goal is now to investigate the action of the left mutation $\L_{\UU_X}$ on a stability condition $\sigma(s_2,q_2)$ on $\Ku(X)_2$ as in Proposition \ref{prop_stabcondKu}(2). The main statement is Corollary \ref{cor_leftmutationstabafterU}.

We would like to apply the same technique as in the previous section. In particular, we need to consider a stability condition corresponding to $(s_2,q_2)$ very close to the point $(-\frac{1}{2}, \frac{1}{20}) \in \partial R_{\frac{3}{20}}$. Thus we have to work with the stability conditions induced from (a tilt of) the tilt stability conditions $\sigma_{s,q}$ below the parabola $q-\frac{1}{2}s^2=0$. 

We start by fixing $s=-\frac{1}{2}+\epsilon$ for $\epsilon>0$ very small such that $s < -\sqrt{\frac{3}{20}}$. For simplicity, we can assume $\epsilon< \frac{1}{10}$. Consider $q>0$ such that the point $(s,q)$ satisfies the conditions in Proposition \ref{prop_stabcondKu}(2). Explicitly, we have that
\begin{equation}
\label{eq_q}
q \in (\frac{1}{20}+\frac{1}{2}\epsilon^2-\frac{1}{2}\epsilon, \frac{1}{20}-\frac{1}{10}\epsilon).  
\end{equation}
On the other side of the vertical wall for $\UU_X$, we consider $s'=-\frac{1}{2}-\epsilon'$ for $\epsilon'>0$ very small. Note that if $q'$ varies in
\begin{equation} \label{eq_q'}
(\frac{1}{20}+\frac{1}{2}(\epsilon')^2+\frac{1}{2}\epsilon', \frac{1}{20}+\frac{9}{10}\epsilon'),   
\end{equation}
then $(s',q')$ satisfies the conditions in Proposition \ref{prop_stabcondKu}(1), so this point induces a stability condition on $\Ku(X)_1$ after suitable tilting.

We fix $\bar{\mu}=-\frac{1}{10}$. Note that
$$\mu_{s,q}(\OO_X)=\frac{-q}{\frac{1}{2}-\epsilon}>-\frac{1}{10} \Longleftrightarrow q < \frac{1}{20}-\frac{1}{10}\epsilon$$
and
$$\mu_{s,q}(\UU_X[1])=\frac{\frac{1}{10}-2q}{-2\epsilon}< -\frac{1}{10} \Longleftrightarrow q < \frac{1}{20}- \frac{1}{10}\epsilon$$
which holds by \eqref{eq_q}.

On the other side, fix $\bar{\mu}'=-\frac{9}{10}$. Similarly, we have
$$\mu_{s',q'}(\OO_X(-H)[1]) < \bar{\mu}' < \mu_{s',q'}(\UU_X).$$

\begin{lemma}
\label{lemma_coh2tilt}
Let $\epsilon >0$ be very small and $\bar{\mu}=-\frac{1}{10}$. Assume that $F \in \emph{Coh}^{\bar{\mu}}_{-\frac{1}{2}+\epsilon, q}(X)$ for every $q$ satisfying \eqref{eq_q}. For $\bar{\mu}'=-\frac{9}{10}$, there exists $\epsilon' >0$, $q'$ satisfying \eqref{eq_q'} such that 
$$F \in \langle \emph{Coh}^{\bar{\mu}'}_{-\frac{1}{2} -\epsilon',q'}(X), \emph{Coh}^{\bar{\mu}'}_{-\frac{1}{2} -\epsilon',q'}(X)[1], \UU_X[2] \rangle.$$
\end{lemma}
\begin{proof}
Fix $s=-\frac{1}{2}+\epsilon$.% and $s'=-\frac{1}{2}-\epsilon'$, $q$ satisfying \eqref{eq_q} and $q'=q+\epsilon$.

\textbf{Step 1:} We claim that if $E \in \Coh^s(X)$ such that $\mu^+_{s,q}(E) \leq \bar{\mu}$, then there exist $ \epsilon'>0$ very small such that $E$ is an extension of objects in $\Coh^{s'}(X)$ where $s':=-\frac{1}{2}-\epsilon'$, and objects of the form $\GG[1]$, where $\GG$ is a slope semistable coherent sheaf with $\mu_H(\GG)=-\frac{1}{2}$ and $\mu_{s,q}^+(\GG[1])= \mu_{s,q}(\UU_X[1])$.

We follow the same argument as in Step 1 of the proof of Lemma \ref{lemma_coh_epsilon}. Up to taking $q \to \frac{1}{20}-\frac{1}{10}\epsilon$, we can assume that $\mu^+_{s,q}(E) \leq \mu_{s,q}(\UU_X[1])$. Up to choosing $\epsilon'$ small enough, we have that $E$ is an extension of objects in $\Coh^{s'}(X)$ and objects of the form $\GG[1]$, where $\GG$ is a slope semistable coherent sheaf with $-\frac{1}{2} \leq \mu_H(\GG) \leq s$ and $\mu_{s,q}^+(\GG[1]) \leq  \mu_{s,q}(\UU_X[1])$. However, by a direct computation or comparing the slopes using Figure \ref{fig_4} we must have $\mu_H(\GG)=-\frac{1}{2}$ and $\mu_{s,q}(\GG[1]) =  \mu_{s,q}(\UU_X[1])$, as we wanted. 

\begin{figure}[htb]
\centering
\begin{tikzpicture}[domain=-6:6]
\draw[->] (-6.3,0) -- (6.3,0) node[right] {$s$};
\draw[->] (0,-2) -- (0, 5.5) node[above] {$q$};
\draw plot (\x,{0.125*\x*\x}) node[above] {$q = \frac{1}{2}s^{2}$};
\coordinate (O) at (0,0);
\node  at (O) {$\bullet$};
\draw (O) node [above right]  {$\mathcal{O}_{X}$};

\draw plot (\x,{0.125*\x*\x -0.3}) node[right] {$q = \frac{1}{2}s^{2}- \frac{3}{40}$};

\coordinate (A) at (-1.5,0);
\draw[red] (A)--(O) ;

\draw[domain=-1.9:-1.5,color=red] plot (\x,{0.125*\x*\x -0.3});

\coordinate (U) at (-1.9,0.17); 
\node at (U) {$\bullet$};  
\draw (U) node [below left]  {$\mathcal{U}_{X}[1]$}; 
  
\draw[red] (U)--(O); 

\draw (U) -- (-1.9,-2);
\draw (A) -- (-1.5,-2);

\coordinate (P) at (-1.6,0.09);
\node at (P) {$\cdot$};

\draw[dashed] (P) -- (-1.6,-2);

\draw[domain=-1.9:1, color=blue] plot (\x, {-0.26*\x -0.33});

\coordinate (G) at (-1.8,-1); 
\node at (G) {$\bullet$};  
\draw (G) node [below left]  {$\mathcal{G}[1]$}; 

\draw[domain=-1.8:-0.5] plot (\x, {5*\x+8});

\end{tikzpicture}

\caption{The dashed arrow is the vertical line $s=-\frac{1}{2}+\epsilon$, passing through the point $(s,q)$ corresponding to the fixed tilt stability. The slope $\mu_{s,q}(\UU_X[1])$ is represented in blue. When $\GG[1]$ corresponds to a point in the region $-\frac{1}{2} \leq s \leq -\frac{1}{2}+\epsilon$ and $q \leq \frac{1}{2}s^2-\frac{3}{40}$, we have that $\mu_{s,q}(\GG[1]) \geq \mu_{s,q}(\UU_X[1])$.  \label{fig_4}}  
\end{figure}
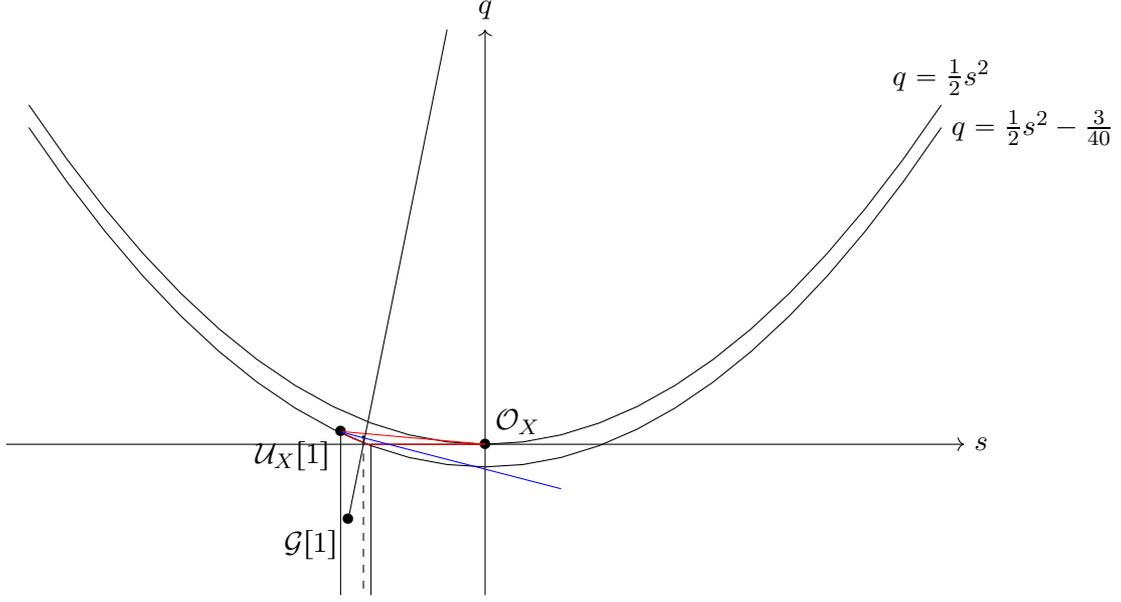

\textbf{Step 2:} We claim that the objects of the form $\GG[1]$ in Step 1 are extensions of copies of $\UU_X[1]$.

Indeed, %we have $\mu_{s',q'}(\UU_X)=\left( \frac{1}{10}-2q' \right) \frac{1}{2\epsilon} > -\frac{9}{10}$ which converges to $-\frac{9}{10}$ for $q' \to \frac{1}{20}+\frac{9}{10}\epsilon$. Thus we can argue as in Step 2 of Lemma \ref{lemma_coh_epsilon} and it remains to show that if $A$ is a stable object in $\Coh^{s'}(X)$ with $\mu_{s',q'}(A)=\mu_{s',q'}(\UU_X)$ and $\frac{\ch_1(A)}{\rk(A)}=-\frac{1}{2}$, then $A \cong \UU_X$. In order to prove this, 
the point $\left( \frac{\ch_1(\GG)}{\rk(\GG)}, \frac{\ch_2(\GG)}{\rk(\GG)} \right)$ belongs to the boundary of $R_{\frac{3}{20}}$. Thus by \cite[Proposition 3.2]{Li_Fano3} $\rk(\GG)$ is either $1$ or $2$. We exclude the case $\rk(\GG)=1$, as the numerical Grothendieck group of $X$ does not contain the class of such object by \cite[Proposition 3.9]{Kuz}. Moreover, $\GG$ is a slope semistable torsion-free sheaf. Since $\Hom(\GG, \OO_X[2])=\Hom(\OO_X(H),\GG[-1])=0$ by stability, \cite[Theorem 3.14]{Lo} implies that $\GG$ is a vector bundle with $\ch(\GG)=\ch(\UU_X)$. It follows that $\GG \cong \UU_X$.

\textbf{Step 3:} We end by showing the statement of the lemma, arguing as in Step 3 of the proof of Lemma \ref{lemma_coh_epsilon}. %$\Coh^{\bar{\mu}}_{s,q}(X) \subset \langle \Coh^{\bar{\mu}'}_{s',q'}(X), \Coh^{\bar{\mu}'}_{s',q'}(X)[1], \UU_X[2] \rangle$ for some $q'$ and $\epsilon$ small enough.

Consider $F \in \Coh^{\bar{\mu}}_{s,q}(X)$. By definition $F$ is an extension of the form
$$A[1] \to F \to B$$
where $B$ (resp.\ $A$) belongs to $\Coh^s(X)_{\mu_{s,q}>\bar{\mu}}$ (resp.\ $\Coh^s(X)_{\mu_{s,q}\leq \bar{\mu}}$).  

By Step 2, we can find $s'$ such that $A[1]$ is an extension of objects in $\Coh^{s'}(X)[1]$ and copies of $\UU_X[2]$. In particular, $A[1] \in \langle \Coh^{s'}(X)[1], \UU_X[2] \rangle$. On the other hand, note that $B \in \langle \Coh^{s'}(X), \Coh^{s'}(X)[1] \rangle$. It is not restrictive to assume that $B$ is $\sigma_{s,q}$-semistable. By \cite[Lemma 3]{LiZhao2}, we have that
$$\mu_{s',q'}(B^+) \geq \mu^+_{s',q'}(B), \quad \mu^-_{s',q'}(B) \geq \mu_{s',q'}(B^-).$$
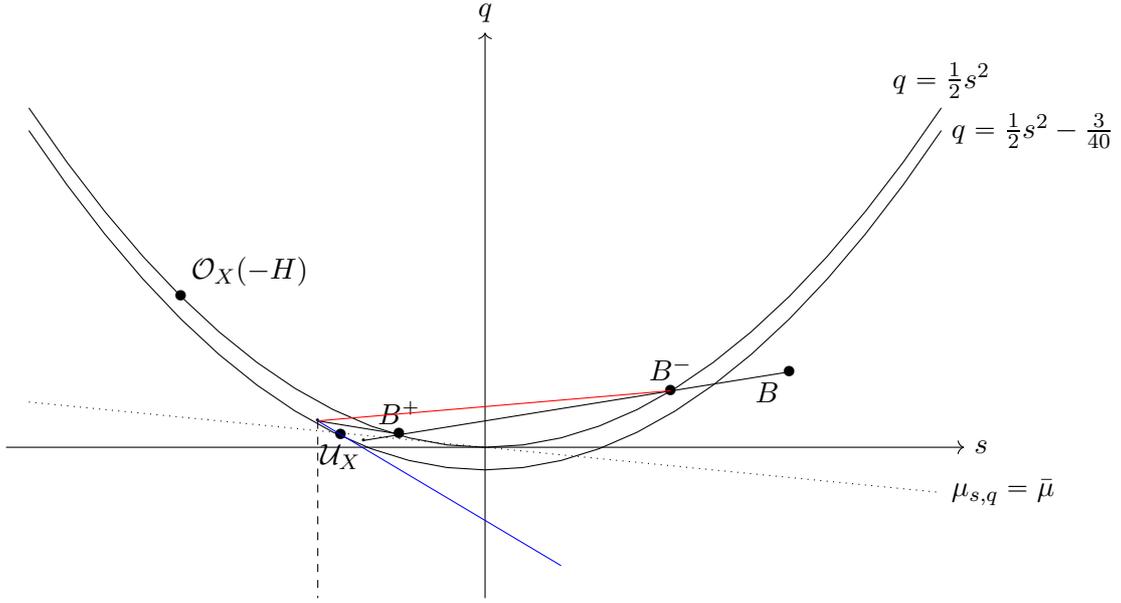
\begin{figure}[htb]
\centering
\begin{tikzpicture}[domain=-6:6]
\draw[->] (-6.3,0) -- (6.3,0) node[right] {$s$};
\draw[->] (0,-2) -- (0, 5.5) node[above] {$q$};
\draw plot (\x,{0.125*\x*\x}) node[above] {$q = \frac{1}{2}s^{2}$};
\coordinate (O) at (0,0);
%\node  at (O) {$\bullet$};
%\draw (O) node [above right]  {$\mathcal{O}_{X}$};

\draw plot (\x,{0.125*\x*\x -0.3}) node[right] {$q = \frac{1}{2}s^{2}- \frac{3}{40}$};

\coordinate (A) at (-1.5,0);
%\draw[red] (A)--(O) ;

%\draw[domain=-2.3:-1.5,color=red] plot (\x,{0.125*\x*\x -0.3});

\coordinate (U) at (-1.9,0.17); 
\node at (U) {$\bullet$};  
\draw (U) node [below]  {$\mathcal{U}_{X}$}; 
  
%\draw[red] (U)--(O); 

\coordinate (OH) at (-4,2);
\node  at (OH) {$\bullet$};
\draw (-4,2) node [above right]  {$\mathcal{O}_{X}(-H)$};

\coordinate (M) at (-2.3,0.38);

\coordinate (P) at (-1.6,0.09);
\node at (P) {$\cdot$};

%\draw[red] (OH)--(M);
%\draw[red] (OH)--(U);

\coordinate (Q) at (-2.2,0.35);
\node at (Q) {$\cdot$};
\draw[dashed] (Q) -- (-2.2,-2);

\draw[dotted] plot (\x,{-0.1*\x}) node[right] {$\mu_{s,q} = \bar{\mu}$};

\coordinate (B) at (4,1);
\node at (B) {$\bullet$};  
\draw (B) node [below left]  {$B$}; 

\draw (B)--(P);

\coordinate (B+) at (-1.13,0.175); 
\node at (B+) {$\bullet$};
\draw (B+) node [above]  {$B^+$};

\coordinate (B-) at (2.44,0.75); 
\node at (B-) {$\bullet$};
\draw (B-) node [above]  {$B^-$};

\draw (Q)--(B+);
\draw[red] (Q)--(B-);

\draw[domain=-2.2:1, color=blue] plot (\x,{-0.6*\x-0.97});

\end{tikzpicture}

\caption{We represent in blue the slope of $\UU_X$ with respect to $(s',q')$, and in red that of $B^-$, where $s'=-\frac{1}{2}-\epsilon'$.  \label{fig_5}} 
\end{figure}

Comparing the slopes as in Figure \ref{fig_5}, we see that 
$$\mu_{s',q'}(B^-) > \mu_{s',q'}(\UU_X) > -\frac{9}{10}=\bar{\mu}'.$$
We deduce that $B \in \langle \Coh^{s'}(X)_{\mu_{s',q' > \bar{\mu}'}}, \Coh^{s'}(X)[1] \rangle$.
%For $\epsilon \to 0$, we have $q \to \frac{1}{20}$ and as in Step 3 of Lemma \ref{lemma_coh_epsilon}, we have 
%$$\mu^-_{s',q'}(B) \to \mu_{s,q}(B) >\bar{\mu}=-\frac{1}{10}> \bar{\mu}'=-\frac{9}{10}.$$
%Thus up to taking a smaller $\epsilon$, we can assume that $\mu^-_{s',q'}(B)> \bar{\mu}'$. This implies 
%$$B \in \langle \Coh^{\bar{\mu}'}_{s',q'}(X), \Coh^{\bar{\mu}'}_{s',q'}(X)[1] \rangle.$$
Putting everything together, we deduce the statement. 
\end{proof}

\begin{lemma}
\label{lemma_leftmutationU}
Fix $\bar{\mu}=-\frac{1}{10}$ and $\epsilon>0$ very small. Then there exists $q$ satisfying \eqref{eq_q} such that if $F \in \CCoh^{\bar{\mu}}_{-\frac{1}{2}+\epsilon,q}(X)$, then $\L_{\UU_X}(F)$ is in $\CCoh^{\bar{\mu}}_{-\frac{1}{2}+\epsilon,q}(X)$.
\end{lemma}
\begin{proof}
Note that $\bar{\mu}=-\frac{\textrm{Re}u}{\textrm{Im}u}$ for $u=\frac{1}{\sqrt{101}}(1+10\sqrt{-1})$. In particular, by definition
$$Z^{\bar{\mu}}_{s,q}(-)=\frac{1}{u}Z_{s,q}(-) = \left( \frac{\sqrt{101}}{101}-  \frac{10\sqrt{101}}{101}\sqrt{-1} \right)Z_{s,q}(-).$$
Set $s=-\frac{1}{2}+\epsilon$. Then $\textrm{Im}Z^{\bar{\mu}}_{s,q}(\UU_X[2])=\frac{\sqrt{101}}{101}(-2\epsilon+1-20q)$ which converges to $0$ as $q \to \frac{1}{20}-\frac{1}{10}\epsilon$. Thus $\mu_{s,q}^{\bar{\mu}}(\UU_X[2]) \to +\infty$ for $q \to \frac{1}{20}-\frac{1}{10}\epsilon$. The same argument of Lemma \ref{lemma_leftmutationO} implies the statement.
\end{proof}

The next results follow from Lemma \ref{lemma_coh2tilt} and Lemma \ref{lemma_leftmutationU}, arguing as in Lemma \ref{lemma_leftmutationheart}, Proposition \ref{prop_leftmutationstab}, Corollary \ref{cor_leftmutationstab}.

\begin{lemma}
\label{lemma_leftmutationheartafterU}
Let $\epsilon>0$ be very small. Then there exist $q$ satisfying \eqref{eq_q}, $\epsilon'>0$ very small and $q'$ satisfying \eqref{eq_q'} such that
$$\L_{\UU_X}(\AA(-\frac{1}{2}+\epsilon,q)) \subset \langle \AA(-\frac{1}{2}-\epsilon',q'), \A(-\frac{1}{2}-\epsilon',q')[1] \rangle.$$
\end{lemma}

\begin{prop}
\label{prop_leftmutationstabafterU}
Let $\epsilon>0$ be very small. Then there exist $q$ satisfying \eqref{eq_q}, $\epsilon'>0$ very small and $q'$ satisfying \eqref{eq_q'} such that there exists $\widetilde{g} \in \widetilde{\emph{GL}}^+_2(\R)$ satisfying
$$\L_{\UU_X} \cdot \sigma(-\frac{1}{2}+\epsilon,q)= \sigma(-\frac{1}{2}-\epsilon',q') \cdot \widetilde{g}.$$
\end{prop}

\begin{cor}
\label{cor_leftmutationstabafterU}
For $i=1, 2$, if $\sigma(s_i,q_i)$ is a stability condition on $\Ku(X)_i$, then there exists $\widetilde{g} \in \widetilde{\emph{GL}}^+_2(\R)$ such that
$$\L_{\UU_X} \cdot \sigma(s_2,q_2)= \sigma(s_1,q_1) \cdot \widetilde{g}.$$
\end{cor} 

\subsection{End of the proof} 
\label{sec_endofproof}

We are now ready to complete the proof of our main result. 

\begin{proof}[Proof of Theorem \ref{thm_mainsec3}]
Let $\sigma(s_3,q_3)$ be a stability condition on $\Ku(X)_3$ as induced in Proposition \ref{prop_stabcondKu}(3). Consider a stability condition $\sigma(s_1,q_1)$ on $\Ku(X)_1$ as in Proposition \ref{prop_stabcondKu}(1) which is above the parabola $q-\frac{1}{2}s^2=0$. By Corollary \ref{cor_leftmutationstab} and Corollary \ref{cor_leftmutationstabafterU} there exists $\widetilde{g} \in \widetilde{\text{GL}}^+_2(\R)$ such that
$$\L_{\UU_X} \cdot \L_{\OO_X} \cdot \sigma(s_3,q_3)= \sigma(s_1,q_1) \cdot \widetilde{g}.$$
Note that if $F \in \Ku(X)_1$ is $\sigma_{s_1,q_1}$-semistable, then $F(H) \in \Ku(X)_3$ is $\sigma_{s_1+1,q_1'}$-semistable for $q_1'=\frac{1}{2}+s_1+q_1$ (see for instance \cite[Proof of Lemma 4]{LiZhao2}). Moreover, $(s_1+1,q_1')$ satisfies the conditions in Proposition \ref{prop_stabcondKu}(3). This implies 
$$\A(s_1,q_1)(H) \subset \langle \A(s_1+1,q_1'), \A(s_1+1,q_1')[1] \rangle.$$
Arguing as in Proposition \ref{prop_leftmutationstab}, it follows that 
\begin{equation}
\label{eq_actionofO(1)}
(- \otimes \OO_X(H)) \cdot \sigma(s_1,q_1)= \sigma(s_1+1,q_1') \cdot \widetilde{f}
\end{equation}
for $\widetilde{f} \in \widetilde{\text{GL}}^+_2(\R)$. Since the action by equivalences and by $\widetilde{\text{GL}}^+_2(\R)$ on the stability manifold commute, by \eqref{eq_defofSerrefunctor} this implies
$$S_{\Ku(X)_3}^{-1} \cdot \sigma(s_3,q_3)= \sigma(s_1+1,q_1') \cdot \widetilde{h}$$
for $\widetilde{h} \in \widetilde{\text{GL}}^+_2(\R)$. By Proposition \ref{prop_sameorbit} we have that $\sigma(s_3,q_3)$ and $\sigma(s_1+1,q_1')$ are in the same orbit with respect to the $\widetilde{\text{GL}}^+_2(\R)$-action, which implies the claim.
\end{proof}

As a consequence, we obtain the same result for the Serre functors of $\Ku(X)_2$ and $\Ku(X)_1$.
\begin{cor} \label{cor_finalresult}
For $i=1,2,3$, let $\sigma(s_i,q_i)$ be a stability condition on $\Ku(X)_i$ as induced in Proposition \ref{prop_stabcondKu}(i). Then there exists $\widetilde{g} \in \widetilde{\emph{GL}}^+_2(\R)$ such that $$S^{-1}_{\Ku(X)_i} \cdot \sigma(s_i,q_i)= \sigma(s_i,q_i) \cdot \widetilde{g}.$$
\end{cor}
\begin{proof}
The case of $i=3$ is Theorem \ref{thm_mainsec3}. For $i=1$ it is enough to note that $(- \otimes \OO_X(H))$ induces an equivalence between $\Ku(X)_1$ and $\Ku(X)_3$. Using the fact that the Serre functors commute with equivalences and \eqref{eq_actionofO(1)}, we deduce the statement for $\Ku(X)_1$. If $i=2$, we apply the same argument since $\L_{\OO_X}$ induces an equivalence between $\Ku(X)_3$ and $\Ku(X)_2$ and using Corollary \ref{cor_leftmutationstab}. 
\end{proof}  

To complete the proof of Theorem \ref{thm_main} it remains to show that the Serre functor shifted by $-2$ acts as the identity on $\sigma(s,q)$. This is done in Corollary \ref{cor_stabonfourfold}.

\subsection{Other Fano threefolds of Picard rank $1$, index $1$}

It is natural to ask whether the above procedure applies to the Kuznetsov component of other Fano threefolds of Picard rank $1$ and index $1$. Recall that there are $10$ deformations types of these Fano threefolds, classified in terms of the genus $g$, which is the positive integer such that the degree $d:=H^3=2g-2$, corresponding to $2 \leq g \leq 12$, $g \neq 11$ \cite{IP}. If $X$ has even genus $g \geq 6$, then by \cite[Lemma 3.6]{Kuz}, \cite[Proposition and Definition 6.3]{BLMS} there is a semiorthogonal decomposition of the form
$$\D(X)= \langle \Ku(X), \EE_2, \OO_X \rangle.$$
Here $\EE_2$ is a vector bundle of rank $2$, obtained by restricting the tautological bundle on a suitable Grassmannian $\G(2, n)$.

Fano threefolds of genus $6$ are GM threefolds. When $X$ has genus $10$, then $\Ku(X) \simeq \D(C_2)$, where $C_2$ is a smooth curve of genus $2$, while if $X$ has genus $12$, then $\Ku(X) \simeq \D(Q_3)$, where $Q_3$ is the Kronecker quiver with three arrows. If $X$ has genus $8$, then $\Ku(X)$ is noncommutative, namely it is not equivalent to the derived category of a variety, and $\Ku(X)$ is equivalent to the Kuznetsov component of a cubic threefold. See \cite[Section 4]{Kuz}.

In all these cases, by \cite{BLMS} the construction reviewed in Section \ref{sec_stabcondonKu} allows to induce stability conditions on $\Ku(X)$ and Theorem \ref{thm_BLMSresult} holds for $\Ku(X)$ replacing $\UU_X$ with $\EE_2$ in the statement. We note the following facts:
\begin{enumerate}
\item We can define $\Ku(X)_i$ for $i=1,2,3$ as in \eqref{eq_Ku1}, \eqref{eq_Ku2}, \eqref{eq_Ku3} and explicit the Serre functor of $\Ku(X)_3$ as in \eqref{eq_defofSerrefunctor}. 
\item Li's stronger Bogomolov inequality holds for slope stable coherent sheaves on $X$ by \cite[Theorem 0.3]{Li_Fano3}, thus every pair $(s,q)$ in the region $R_{\frac{3}{2d}}$ defines a weak stability condition $\sigma_{s,q}$ on $\D(X)$.
\item Since $\ch_{\leq 2}(\EE_2)=(2, -H, \frac{g-4}{2d}H^2)$, the point $(-\frac{1}{2}, \frac{d-6}{8d})$ defined by $\EE_2$ belongs to the parabola $s^2-2q=\frac{3}{2d}$. In particular, $\EE_2$ is $\sigma_{s,q}$-stable for every $(s, q) \in R_{\frac{3}{2d}}$, and thus we can induce stability conditions on $\Ku(X)$ by restriction of a tilting of $\sigma_{s,q}$ for $(s,q)$ as in Proposition \ref{prop_stabcondKu}. 
\item By \cite{Kuz} the numerical Grothendieck group of $\Ku(X)$ has rank $2$ and a basis is given by
\begin{align*}
b_1&=1 -\frac{d+2}{4d}H^2 + \dots,\\
b_2&=H -\frac{3d-6}{4d}H^2+ \dots  
\end{align*}
It is not hard to check that the basis $Z(s,q)(b_1)$, $Z(s,q)(b_2)$ of $\C$ have the same orientation for every $(s,q)$ as in Proposition \ref{prop_stabcondKu}(1), using that $(s,q)$ is above the parabola $q= \frac{1}{2}s^2-\frac{3}{4d}$.
\end{enumerate}
Besides this, the argument explained in Section \ref{sec_proof} does not use anything else specific of working with a GM threefold. In fact, we have decided to consider GM threefolds, since this is the case we are interested for applications and in order to make a more readable proof. We obtain the following generalization of Theorem \ref{thm_main} and Corollary \ref{cor_finalresult}.

\begin{thm} 
\label{thm_fanoevengenus}
Let $X$ be a Fano threefold of Picard rank $1$, index $1$ and even genus $g \geq 6$. For $i=1,2,3$, let $\sigma(s_i,q_i)$ be a stability condition on $\Ku(X)_i$ as induced in Proposition \ref{prop_stabcondKu}(i). Then there exists $\widetilde{g} \in \widetilde{\emph{GL}}^+_2(\R)$ such that $$S^{-1}_{\Ku(X)_i} \cdot \sigma(s_i,q_i)= \sigma(s_i,q_i) \cdot \widetilde{g}.$$
\end{thm}

\begin{rmk} \label{rmk_quest}
If $Y_d$ is a Fano threefold of Picard rank $1$, index $2$ and degree $d$, by \cite{Kuz} there is a semiorthogonal decomposition of the form
$$\D(X)= \langle \Ku(Y_d), \OO_{Y_d}, \OO_{Y_d}(1) \rangle$$
and by \cite{BLMS} there are stability conditions on $\Ku(Y_d)$, induced by restriction of a double tilting of slope stability on $\D(X)$. Denote by $\MM^i_d$ the moduli space of Fano threefolds of index $i$ and degree $d$ for $i=1,2$. By \cite[Theorem 3.8]{Kuz}, for $d=3, 4, 5$ there is a correspondence $\ZZ_d \subset \MM^2_d \times \MM^1_{4d+2}$, dominant over each factor, such that for every point $(Y_d, X_{4d+2}) \in \ZZ_d$, there is an equivalence 
$$\Phi_d \colon \Ku(Y_d) \simeq \Ku(X_{4d+2}).$$

Via $\Phi_d$, the stability conditions $\sigma(s,q)$ on $\Ku(X_{4d+2})$ define stability conditions on $\Ku(Y_d)$. By Theorem \ref{thm_fanoevengenus} and \cite[Theorem 3.2]{FP}, \cite[Theorem 4.25]{JLLZ} these stability conditions are in the same orbit with respect to the $\widetilde{\text{GL}}^+_2(\R)$-action of those constructed in \cite{BLMS} on $\Ku(Y_d)$. 

An interesting problem would be to understand whether there is a unique orbit of stability conditions on $\Ku(Y_d)$. This observation could be an evidence towards a positive answer to this question.
\end{rmk}

\begin{rmk}
In the odd genus cases we have the following semiorthogonal decompositions by \cite{Kuz_hyperplane}:
$$\D(X_{12})=\langle \Ku(X_{12}), \EE_5, \OO \rangle, \quad \D(X_{16})=\langle \Ku(X_{16}), \EE_3, \OO \rangle,$$
where $X_{12}$ has degree $12$, genus $7$ and $X_{16}$ has degree $16$, genus $9$. Here $\EE_5$ and $\EE_3$ are vector bundles of rank $5$ and $3$, respectively. The Chern characters of $\EE_5$ and $\EE_3$ do not define points on the parabola $s^2-2q=\frac{3}{2d}$ for $d=12$ and $d=16$, respectively. Thus in order to generalize the argument of Theorem \ref{thm_mainsec3} to these cases, one needs first to control the tilt stability of $\EE_5$ and $\EE_3$. On the other hand, in these cases the Kuznetsov component is equivalent to the bounded derived category of a curve of genus $\geq 1$. Thus by \cite{Macri} there is a unique orbit of stability conditions with respect to the $\widetilde{\text{GL}}^+_2(\R)$-action, and thus all the stability conditions are preserved by the Serre functor up to the $\widetilde{\text{GL}}^+_2(\R)$-action.
\end{rmk}

\section{Serre-invariant stability conditions} \label{sec_applications}
In this section, we drop the superfluous subscript and write $\Ku (X)= \Ku(X)_i$ for any given $i=1,2,3$ to simplify the notation, as the results contained herein hold for all such choices. We introduce the following definition (see \cite[Definition 3.1]{FP}).
 \begin{dfn} \label{def_serreinvariant}
 A stability condition $\sigma$ on $\Ku(X)$ is Serre-invariant, or $S_{\Ku(X)}$-invariant, if $S_{\Ku(X)} \cdot \sigma = \sigma \cdot \widetilde{g}$ for some $\widetilde{g} \in \widetilde{\text{GL}}^+_2(\R)$.
 \end{dfn}
\noindent In Theorem \ref{thm_mainsec3}, we have established that the stability conditions on $\Ku(X)$ as in Proposition 3.2 are $S_{\Ku(X)}$-invariant. We now aim to explore the implications that this fact has for the existence of Bridgeland stability conditions on special GM fourfolds (Corollary \ref{cor_stabonfourfold}) and to show that there is a unique orbit with respect to the $\widetilde{\text{GL}}^+(2, \R)$-action of $S_{\Ku(X)}$-invariant stability conditions.  

\subsection{Stability conditions on special GM fourfolds}
We begin by setting up some notation. Let $Y$ be a variety with a line bundle $\mathcal{O}_Y(1)$. We say that $\D(Y)$ admits a rectangular Lefschetz decomposition with respect to $\mathcal{O}_Y(1)$ if there is an admissible subcategory $\mathcal{B} \hookrightarrow \D(Y)$ such that 
\begin{align}
    \D(Y) = \langle \mathcal{B}, \mathcal{B}(1), \cdot \cdot \cdot, \mathcal{B}(m-1) \rangle
\end{align}
is a semiorthogonal decomposition for some integer $m$. Given such a decomposition of $\D(Y)$, pick $n,d \in \mathbb{N}$ such that $nd \leq m$. Suppose we have a degree-$n$ cyclic cover $f: X \to Y$ of $Y$ ramified in a Cartier divisor $Z$ in the linear system corresponding to $\mathcal{O}_Y (nd)$. If $i: Z \hookrightarrow Y$ is the inclusion, then the derived pullbacks $i^*$ and $f^*$ are fully faithful upon restriction to $\mathcal{B}$. We obtain semiorthogonal decompositions
\begin{align}
    \D(X) &= \langle \mathcal{A}_X, f^* \mathcal{B}, \cdot \cdot \cdot, f^* \mathcal{B}(m-(n-1)d-1) \rangle,  \\
    \D(Z) &= \langle \mathcal{A}_Z, i^* \mathcal{B}, \cdot \cdot \cdot, i^* \mathcal{B}(m-nd-1) \rangle, 
\end{align}
with $\mathcal{A}_X = \langle f^* \mathcal{B}, \cdot \cdot \cdot, f^* \mathcal{B}(m-(n-1)d-1) \rangle^\perp$ and $\mathcal{A}_Z$ defined similarly. The following theorem of Kuznetsov and Perry relates $\mathcal{A}_X$ and  $\mathcal{A}_Z$ in the above scenario. 
\begin{thm}[\cite{KP_cyclic}, Theorem 1.1] \label{thm_equivariantcat} In the setup above, there are fully faithful functors $\Phi_k :\mathcal{A}_Z \to \mathcal{A}_X^{\mu_n} $ for $0 \leq k \leq n-2$ such that there is a semiorthogonal decomposition:
\begin{align}
    \mathcal{A}_X^{\mu_n} = \langle \Phi_0 (\mathcal{A}_Z), \cdot \cdot \cdot, \Phi_{n-2} (\mathcal{A}_Z) \rangle.
\end{align}
\end{thm}
\noindent Here, $\mu_n$ is the group of $n^{th}$ roots of unity, acting on $X$ via automorphisms over $Y$ and $\mathcal{A}_X^{\mu_n}$ is the corresponding equivariant category.

If we now assume that $X$ is a special GM fourfold, then the map $X \to \textrm{Gr}(2,5)$ is a double cover of its image $Y$, ramified over an ordinary GM threefold $Z \hookrightarrow Y$. In the notation of \cite{KP_cyclic}, we have $n=2$, $d=1$, $e=2$ and $\AA_X=\Ku(X)$, $\AA_Z=\Ku(Z)$ are the Kuznetsov components of the GM fourfold and threefold.  By Theorem \ref{thm_equivariantcat}, the map $\Phi_0$ provides an equivalence of categories $\textrm{Ku}(Z) \cong \textrm{Ku}(X)^{\mu_2}$. As shown in \cite[Corollary 1.3, Proposition 7.10]{KP_cyclic}, which makes use of \cite{E}, we have dually an equivalence
\begin{equation}
\label{eq_equivalence}
\textrm{Ku}(Z)^{\mathbb{Z}/2\mathbb{Z}} \cong \textrm{Ku}(X).    
\end{equation}
The action of $\Z /2\Z$ on $\Ku(Z)$ is induced by the rotation functor $\L_{i^*\BB}(- \otimes \OO_X(H))[-1]$, where $i^*\BB= ^{\perp}\Ku(Z)$. Using this equivalence and Theorem \ref{thm_mainsec3}, we have the following result.

\begin{cor} \label{cor_stabonfourfold}
Let $X$ be a special GM fourfold and $Z$ be its associated ordinary GM threefold. The stability conditions $\sigma(s,q)$ on $\Ku(Z)$ defined in Proposition \ref{prop_stabcondKu} induce stability conditions on the equivariant category $\Ku(Z)^{\mathbb{Z}/2\mathbb{Z}}$. In particular, they define stability conditions on $\Ku(X)$.
\end{cor} 
\begin{proof}
It is sufficient to prove that $S_{\textrm{Ku}(Z)}[-2] \cdot \sigma (s,q) = \sigma (s,q)$, since the $\mathbb{Z}/2 \mathbb{Z}$-action on $\textrm{Ku}(Z)$ is induced by $\L_{i^*\BB}(- \otimes \OO_X(H))[-1]=S^{-1}_{\Ku(Z)}[2]$, or equivalently by $S_{\textrm{Ku}(Z)}[-2]$. By Theorem \ref{thm_mainsec3}, there is some $\widetilde{g}= (M,g) \in  \widetilde{\mathrm{GL}}^+_2(\R)$ such that $S_{\textrm{Ku}(Z)}[-2] \cdot  \sigma(s,q) = \sigma(s,q) \cdot \widetilde{g}$. Applying the involution $S_{\textrm{Ku}(Z)}[-2]$ to both sides of this equality yields: $\sigma (s,q) = \sigma(s,q) \cdot \widetilde{g}^2$. Writing $\sigma (s,q) = (\mathcal{P},Z)$, at the level of slicings this gives $\mathcal{P}(\phi) = \mathcal{P}(g^2 (\phi))$ for any $\phi \in \mathbb{R}$, hence $g: \mathbb{R} \to \mathbb{R}$ is an increasing involution, so we must have $g= \textrm{id}$. On the other hand, on central charges we have $M^{-2} \circ Z = Z$. The image of $Z$ is not contained in a line, hence $M^{-2}$ agrees with the identity on two linearly independent vectors in $\mathbb{C} \cong \mathbb{R}^2$, thus $M^2 = \textrm{id}$. There are only three conjugacy classes of $2 \times 2$ matrices over $\mathbb{R}$ squaring to the identity, one of which has negative determinant, hence $M = \pm I$. We cannot have $M= -I$, since $M$ induces the identity on the circle, thus $M=I$ and we deduce that $S_{\textrm{Ku}(Z)}[-2] \cdot \sigma (s,q) = \sigma (s,q)$ as claimed. 

As a consequence, if $\text{Forg} \colon \Ku(Z)^{\mathbb{Z}/2\mathbb{Z}} \to \Ku(Z)$ denotes the forgetful functor, then by \cite[Lemma 2.16]{MMS} we have that $\Forg^{-1} \cdot \sigma(s,q)$ defines a stability condition on $\Ku(Z)^{\mathbb{Z}/2\mathbb{Z}}$. Composing with the equivalence in \eqref{eq_equivalence} we obtain stability conditions on $\Ku(X)$.
\end{proof}

\begin{rmk} \label{rmk_stabonfourfold}
Note that the above proof does not use anything specific on the stability conditions $\sigma(s,q)$. In particular, Corollary \ref{cor_stabonfourfold} holds more generally for every Serre-invariant stability conditions on $\Ku(Z)$.  
\end{rmk}
 
\subsection{Uniqueness}

Let $X$ be a GM threefold. The aim of this section is to prove the following result.

\begin{cor} \label{cor_uniqueness}
If $\sigma_1$, $\sigma_2$ are $S_{\Ku(X)}$-invariant stability conditions, then there exists $\widetilde{g} \in \widetilde{\emph{GL}}^+_2(\R)$ such that $\sigma_2= \sigma_1 \cdot \widetilde{g}$.
\end{cor}

Corollary \ref{cor_uniqueness} has been recently proved in \cite[Lemmas 4.22, 4.23, 4.24]{JLLZ}. Here we give an alternative proof making use of the following result obtained from \cite{FP}.

\begin{thm}[\cite{FP}, Theorem 3.2, Lemma 3.6] \label{criterion}
Let $\TT$ be a $\C$-linear triangulated category of finite type whose Serre functor satisfies $S_{\TT}^2=[4]$ and whose numerical Grothendieck group $\NN(\TT)$ has rank $2$. Assume further the following conditions hold:
\begin{enumerate}
\item $\ell_{\TT} := \max\{\chi(v, v) \colon 0 \neq v \in \mathcal{N}(\TT) \} < 0$.
\item 
There are three objects $Q_1, Q_2, Q_2' \in \TT$ such that $Q_2$ and $Q_2'$ have the same class in $\mathcal{N}(\TT)$, $Q_1$ is not isomorphic to $Q_2$, or $Q_2'[1]$, and
\begin{align*}
&-\ell_{\TT} +1 \leq  \hom^1(Q_i, Q_i),\, \hom^1(Q_2', Q_2') <  -2\ell_{\TT} +2,\\
&\hom(Q_2, Q_1) \neq 0\\
&\hom(Q_1, Q_2'[1]) \neq 0\\
&\hom(Q_2', Q_2[3]) = 0.
\end{align*} 
\end{enumerate}
Then there exists a unique orbit of $S_{\TT}$-invariant stability conditions on $\TT$ with respect to the $\widetilde{\emph{GL}}^+_2(\R)$-action.
\end{thm}

Let us check the conditions of Theorem \ref{criterion} for the Kuznetsov component $\Ku(X):=\langle \UU_X, \OO_X \rangle^\perp$ of a GM threefold $X$. We have already recalled in Section \ref{sec_GM} that $S_{\Ku(X)}^2=[4]$ and $\NN(\Ku(X))$ has rank $2$. By \cite{Kuz} the basis $b_1$, $b_2$ of \eqref{eq_basis} has intersection form
$$\begin{pmatrix}
-2 & -3 \\
-3 & -5
\end{pmatrix}.$$
For $v= \alpha b_1+\beta b_2 \in \NN(\Ku(X))$, we have 
$$v^2=-2\alpha^2 - 6\alpha \beta - 5\beta^2=-(\alpha+2\beta)^2-(\alpha+\beta)^2 \leq -1,$$
so $\ell_{\Ku(X)}=-1$. To find the suitable objects $Q_i$, we argue similarly as in \cite[Lemma 4.26]{JLLZ}, just using conics instead of lines. Let $C \subset X$ be a smooth conic. Its ideal sheaf $\II_C$ is in $\langle \OO_X \rangle^\perp$ and the left mutation $\L_{\UU_X}(\II_C)$ is in $\Ku(X)$ by definition, sitting in the triangle
\begin{equation} \label{eq_mutconic}
\UU_X \to \II_C \to \L_{\UU_X}(\II_C). 
\end{equation}
The latter can be computed using the short exact sequence
$$0 \to \II_C \to \OO_X \to \OO_C \to 0$$ 
and $h^0(\UU_X^\vee|_C)=4$, $h^i(\UU_X^\vee|_C)=0$ for $ i \neq 0$, $h^0(\UU_X^\vee)=5$, $h^i(\UU_X^\vee)=0$ for $ i \neq 0$. Assume further that $C$ is generic on $X$. Consider now a smooth twisted cubic $D' \subset X$ such that $D'$ does not intersect $C$ and its ideal sheaf $\II_{D'} \in \Ku(X)$. Note that the generic twisted cubic $D'$ satisfies these conditions. Finally, pick a twisted cubic $D \subset X$ such that $\II_{D} \in \Ku(X)$ and $C$ is an irreducible component of $D$. The existence of such $D$ has been proved in \cite[Lemma 4.24]{JLLZ}. Set
$$Q_1:=\L_{\UU_X}(\II_C), \quad Q_2:=\II_{D}, \quad Q_2':=\II_{D'}.$$ 
Clearly $Q_2$ and $Q_2'$ have the same class in $\NN(\TT)$ and they are not isomorphic to $Q_1$. The following lemma ends the proof of Corollary \ref{cor_uniqueness}.

\begin{lemma}
With the notation above, we have
\begin{align*}
&\hom^1(Q_1,Q_1)=2, \quad \hom^1(Q_2,Q_2)=\hom^1(Q_2', Q_2')=3 ,\\
&\hom(Q_2, Q_1) \neq 0 \quad \hom(Q_1, Q_2'[1]) \neq 0, \quad \hom(Q_2', Q_2[3]) = 0.
\end{align*}
\end{lemma}
\begin{proof}
Note that $\hom^1(Q_1,Q_1)=\hom^1(\II_C, Q_1)$ as $Q_1 \in \Ku(X)$. By Serre duality, we have $\hom^i(\II_C, \UU_X)=\hom^{3-i}(\UU_X^\vee, \II_C)=\hom^{2-i}(\UU_X^\vee, \OO_C)=h^{2-i}(\UU_X|_C)=0$ for every $i$, as $C$ is a generic conic. Thus $\hom^1(\II_C, Q_1)=\hom^1(\II_C,\II_C)=2$ (see \cite[Lemma 4.2.1(ii), Proposition 4.2.5(iii)]{IP}). With a similar computation as in \cite[Proposition 3.8]{Zhang}, we get $\hom^1(Q_2,Q_2)=\hom^1(Q_2', Q_2')=3$. By Serre duality, $\hom^3(Q_2',Q_2)=\hom(Q_2, Q_2'(-H))=0$ by slope stability of $Q_2$ and $Q_2'(-H)$. Now note that 
$$\hom(Q_2, \UU_X)=\hom(\OO_{D}, \UU_X[1])=\hom(\UU_X, \OO_{D}(-H)[2])=h^2(\UU_X|_{D})=0$$
since $h^i(\UU_X)=0$ for every $i$ and Serre duality. It follows that the space $\Hom(Q_2, \II_C)$ has an injection in $\Hom(Q_2,Q_1)$. Since $C$ is a component of $D$, the former is not $0$ and we get $\hom(Q_2,Q_1) \neq 0$. Finally, we have
$$
\hom(Q_1, Q_2'[1])=\hom(\II_C, \II_{D'}[1])=\hom(\OO_C, \II_{D'}[2])= \hom(\OO_C, \OO_X[2])=h^1(\OO_C(-H))=1,  $$
where in the first and second equality we have used $\II_{D'} \in \Ku(X)$ and in the third the fact that $C \cap D' = \emptyset$. 
\end{proof} 
 
%Second since $[Q_2]=b_1 \in \NN(\TT)$, we have
%$$-2= \hom(Q_2,Q_2) - \hom^1(Q_2,Q_2) +\hom^2(Q_2, Q_2)- \hom^3(Q_2,Q_2).$$ 
%By Serre duality, $\hom^3(Q_2,Q_2)=\hom(Q_2, Q_2(-H))=0$ and $\hom(Q_2,Q_2)=1$ by slope stability of $Q_2$ and $Q_2(-H)$. Since $Q_2 \in \Ku(X)$, we have $\hom^2(Q_2,Q_2)=\hom(Q_2[2],S_{\Ku(X)}(Q_2))$. By  
\bibliographystyle{alpha}
\bibliography{references} 

\newcommand{\etalchar}[1]{$^{#1}$}
\begin{thebibliography}{BMMS12}

\bibitem[AB13]{Arcara_Bertram}
D.~Arcara and A.~Bertram.
\newblock Bridgeland stable moduli spaces for k-trivial surfaces.
\newblock {\em J. Eur. Math. Soc.}, 15:1--38, 2013.

\bibitem[BBD82]{BBD}
A.~A. Beilinson, J.~Bernstein, and P.~Deligne.
\newblock Faisceaux pervers.
\newblock In {\em Analysis and topology on singular spaces, {I} ({L}uminy,
  1981)}, volume 100 of {\em Ast\'erisque}, pages 5--171. Soc. Math. France,
  Paris, 1982.

\bibitem[BBF{\etalchar{+}}20]{Soheyla}
Arend Bayer, Sjoerd Beentjes, Soheyla Feyzbakhsh, Georg Hein, Diletta
  Martinelli, Fatemeh Rezaee, and Benjamin Schmidt.
\newblock The desingularization of the theta divisor of a cubic threefold as a
  moduli space, arXiv:2011.12240, 2020.

\bibitem[BLM{\etalchar{+}}21]{BLMNPS}
Arend Bayer, Mart\'i Lahoz, Emanuele Macr\`\i, Howard Nuer, Alexander Perry,
  and Paolo Stellari.
\newblock Stability conditions in families.
\newblock {\em Publ. Math. Inst. Hautes \'{E}tudes Sci.}, 133:157--325, 2021.

\bibitem[BLMS17]{BLMS}
Arend Bayer, Mart\'i Lahoz, Emanuele Macr{\`{\i}}, and Paolo Stellari.
\newblock Stability conditions on {K}uznetsov components.
\newblock {\em (Appendix joint also with X. Zhao), to appear in: Ann.\ Sci.\
  {E}cole. Norm.\ Sup\'er.\, arXiv:1703.10839.}, 2017.

\bibitem[BMMS12]{BMMS}
M.~Bernardara, E.~Macr\`\i, S.~Mehrotra, and P.~Stellari.
\newblock A categorical invariant for cubic threefolds.
\newblock {\em Adv. Math.}, 229(2):770--803, 2012.

\bibitem[BMS16]{BMS}
Arend Bayer, Emanuele Macr\`\i, and Paolo Stellari.
\newblock The space of stability conditions on abelian threefolds, and on some
  {C}alabi-{Y}au threefolds.
\newblock {\em Invent. Math.}, 206(3):869--933, 2016.

\bibitem[BMSZ17]{BMSZ}
M.~Bernardara, E.~Macr\`\i, B.~Schmidt, and X.~Zhao.
\newblock Bridgeland stability conditions on {F}ano threefolds.
\newblock {\em \'{E}pijournal Geom. Alg\'{e}brique}, 1, 2017.

\bibitem[BMT14]{BMT}
A.\ Bayer, E.\ Macr\`\i, and Y.~Toda.
\newblock Bridgeland stability conditions on threefolds i: Bogomolov-{G}iesker
  type inequalities.
\newblock {\em J. Algebraic Geom.}, 23:117--163, 2014.

\bibitem[Bon89]{Bondal}
A.~I. Bondal.
\newblock Representations of associative algebras and coherent sheaves.
\newblock {\em Izv. Akad. Nauk SSSR Ser. Mat.}, 53(1):25--44, 1989.

\bibitem[Bri07]{Bridgeland07}
Tom Bridgeland.
\newblock Stability conditions on triangulated categories.
\newblock {\em Ann. of Math}, 166:317--345, 2007.

\bibitem[Bri08]{Bridgeland}
Tom Bridgeland.
\newblock Stability conditions on {$K3$} surfaces.
\newblock {\em Duke Math. J.}, 141(2):241--291, 2008.

\bibitem[Ela15]{E}
Alexey Elagin.
\newblock On equivariant triangulated categories, arXiv:1403.7027v2, 2015.

\bibitem[FP21]{FP}
Soheyla Feyzbakhsh and Laura Pertusi.
\newblock Serre-invariant stability conditions and ulrich bundles on cubic
  threefolds, arXiv:2109.13549, 2021.

\bibitem[Gus82]{Gushel}
N.~P. Gushel.
\newblock Fano varieties of genus {$6$}.
\newblock {\em Izv. Akad. Nauk SSSR Ser. Mat.}, 46(6):1159--1174, 1343, 1982.

\bibitem[HRS96]{HRS}
Dieter Happel, Idun Reiten, and Sverre~O. Smal\o.
\newblock Tilting in abelian categories and quasitilted algebras.
\newblock {\em Mem. Amer. Math. Soc.}, 120(575):viii+ 88, 1996.

\bibitem[IP99]{IP}
V.~A. Iskovskikh and Yu.~G. Prokhorov.
\newblock Fano varieties.
\newblock In {\em Algebraic geometry, {V}}, volume~47 of {\em Encyclopaedia
  Math. Sci.}, pages 1--247. Springer, Berlin, 1999.

\bibitem[JLLZ22]{JLLZ}
Augustinas Jacovskis, Zhiyu Liu, Xun Lin, and Shizhuo Zhang.
\newblock Categorical {T}orelli theorems for {G}ushel--{M}ukai threefolds,
  arXiv:2108.02946v2, 2022.

\bibitem[KP17]{KP_cyclic}
Alexander Kuznetsov and Alexander Perry.
\newblock Derived categories of cyclic covers and their branch divisors.
\newblock {\em Selecta Math. (N.S.)}, 23(1):389--423, 2017.

\bibitem[KP18]{KP}
Alexander Kuznetsov and Alexander Perry.
\newblock Derived categories of {G}ushel-{M}ukai varieties.
\newblock {\em Compos. Math.}, 154(7):1362--1406, 2018.

\bibitem[KP19]{KP_cones}
Alexander Kuznetsov and Alexander Perry.
\newblock Categorical cones and quadratic homological projective duality.
\newblock {\em To appear in ASENS, arXiv:1902.09824}, 2019.

\bibitem[Kuz04]{Kuz_cubic}
A.~G. Kuznetsov.
\newblock Derived category of a cubic threefold and the variety {$V_{14}$}.
\newblock {\em Tr. Mat. Inst. Steklova}, 246(Algebr. Geom. Metody, Svyazi i
  Prilozh.):183--207, 2004.

\bibitem[Kuz06]{Kuz_hyperplane}
A.~G. Kuznetsov.
\newblock Hyperplane sections and derived categories.
\newblock {\em Izv. Ross. Akad. Nauk Ser. Mat.}, 70(3):23--128, 2006.

\bibitem[Kuz09]{Kuz}
A.~G. Kuznetsov.
\newblock Derived categories of {F}ano threefolds.
\newblock {\em Tr. Mat. Inst. Steklova}, 264(Mnogomernaya Algebraicheskaya
  Geometriya):116--128, 2009.

\bibitem[Kuz10]{Kuz_cubicfourfold}
Alexander Kuznetsov.
\newblock Derived categories of cubic fourfolds.
\newblock {\em Cohomological and Geometric Approaches to Rationality Problems},
  282:219--243, 2010.

\bibitem[Kuz19]{Kuz_calabi}
Alexander Kuznetsov.
\newblock Calabi-{Y}au and fractional {C}alabi-{Y}au categories.
\newblock {\em J. Reine Angew. Math.}, 753:239--267, 2019.

\bibitem[Li19a]{Li}
Chunyi Li.
\newblock On stability conditions for the quintic threefold.
\newblock {\em Invent. Math.}, 218(1):301--340, 2019.

\bibitem[Li19b]{Li_Fano3}
Chunyi Li.
\newblock Stability conditions on {F}ano threefolds of {P}icard number 1.
\newblock {\em J. Eur. Math. Soc. (JEMS)}, 21(3):709--726, 2019.

\bibitem[Liu19]{Liu}
Y.~Liu.
\newblock Stability conditions on product varieties.
\newblock {\em arXiv:1907.09326. to appear in Crelle.}, 2019.

\bibitem[LM16]{Lo}
Jason Lo and Yogesh More.
\newblock Some examples of tilt-stable objects on threefolds.
\newblock {\em Comm. Algebra}, 44(3):1280--1301, 2016.

\bibitem[LZ19a]{LiZhao_birational}
Chunyi Li and Xiaolei Zhao.
\newblock Birational models of moduli spaces of coherent sheaves on the
  projective plane.
\newblock {\em Geom. Topol.}, 23(1):347--426, 2019.

\bibitem[LZ19b]{LiZhao2}
Chunyi Li and Xiaolei Zhao.
\newblock Smoothness and {P}oisson structures of {B}ridgeland moduli spaces on
  {P}oisson surfaces.
\newblock {\em Math. Z.}, 291(1-2):437--447, 2019.

\bibitem[LZ21]{LZ}
Zhiyu Liu and Shizhuo Zhang.
\newblock A note on {B}ridgeland moduli spaces and moduli spaces of sheaves on
  $x_{14}$ and $y_3$, arXiv:2106.01961, 2021.

\bibitem[Mac07]{Macri}
Emanuele Macr\`\i.
\newblock Stability conditions on curves.
\newblock {\em Math. Res. Letters}, 14:657--672, 2007.

\bibitem[Mac14]{Macri14}
Emanuele Macr\`\i.
\newblock A generalized {B}ogomolov-{G}ieseker inequality for the
  three-dimensional projective space.
\newblock {\em Algebra Number Theory}, 8:173--190, 2014.

\bibitem[MMS09]{MMS}
Emanuele Macr\`\i, Sukhendu Mehrotra, and Paolo Stellari.
\newblock Inducing stability conditions.
\newblock {\em J. Algebraic Geom.}, 18(4):605--649, 2009.

\bibitem[Muk89]{Mukai}
Shigeru Mukai.
\newblock Biregular classification of {F}ano {$3$}-folds and {F}ano manifolds
  of coindex {$3$}.
\newblock {\em Proc. Nat. Acad. Sci. U.S.A.}, 86(9):3000--3002, 1989.

\bibitem[Oka06]{Okada}
S.~Okada.
\newblock Stability manifold of $\mathbb{P}^1$.
\newblock {\em J. Algebraic Geom.}, 15:487--505, 2006.

\bibitem[PPZ19]{PPZ}
Alexander Perry, Laura Pertusi, and Xiaolei Zhao.
\newblock Stability conditions and moduli spaces for {K}uznetsov components of
  {G}ushel-{M}ukai varieties.
\newblock {\em To appear in Geometry and Topology, arXiv:1912.06935}, 2019.

\bibitem[PPZon]{PPZ_preparation}
Alexander Perry, Laura Pertusi, and Xiaolei Zhao.
\newblock Enriques categories, in preparation.

\bibitem[PY20]{PY}
Laura Pertusi and Song Yang.
\newblock Some remarks on {F}ano threefolds of index two and stability
  conditions.
\newblock {\em To appear in IMRN, arXiv:2004.02798}, 2020.

\bibitem[Tod09]{Toda}
Yukinobu Toda.
\newblock Limit stable objects on {C}alabi-{Y}au 3-folds.
\newblock {\em Duke Math. J.}, 149(1):157--208, 2009.

\bibitem[Zha20]{Zhang}
Shizhuo Zhang.
\newblock Bridgeland moduli spaces for {G}ushel--{M}ukai threefolds and
  {K}uznetsov's {F}ano threefold conjecture, arXiv:2012.12193, 2020.

\end{thebibliography}

Dipartimento di Matematica ``F.\ Enriques'', Universit\`a degli Studi di Milano, Via Cesare Saldini 50, 20133 Milano, Italy. \\
\indent E-mail address: \texttt{laura.pertusi@unimi.it}\\
\indent URL: \texttt{http://www.mat.unimi.it/users/pertusi} \\

Department of Mathematics, South Hall, Room 6607, University of California, Santa Barbara, CA 93106, USA. \\
\indent E-mail address: \texttt{robinett@math.ucsb.edu}\\

\end{document}